\documentclass[11pt]{amsart}
\usepackage{amsmath}
\usepackage{amsfonts}
\usepackage{amssymb}
\usepackage{mathrsfs}
\usepackage{todonotes}
\usepackage{graphicx}%
\providecommand{\U}[1]{\protect\rule{.1in}{.1in}}
\newtheorem{theorem}{Theorem}

\newtheorem{claim}[theorem]{Claim}

\newtheorem{corollary}[theorem]{Corollary}

\newtheorem{definition}[theorem]{Definition}

\newtheorem{lemma}[theorem]{Lemma}

\newtheorem{problem}[theorem]{Problem}
\newtheorem{proposition}[theorem]{Proposition}

\title{Are scales Fr\'{e}chet?}

\author{R. Figueroa-Sierra
\and O. Guzm\'{a}n
\and M. Hru\v{s}\'{a}k
\and A. Kwela
}

\thanks{
The second author was supported by the PAPIIT grant
IA 104124 and the CONAHCYT grant CBF2023-2024-903. The third author was supported by a PAPIIT grant IN101323 and a SECIHTI grant CBF-2025-I-898.}

\begin{document} 

\begin{abstract}
We continue the study of Dow spaces of a $\mathfrak{b}$-scale, originally
introduced by Dow in \cite{PiWeightandFrechetProperty}. We prove that it is
consistent that all such spaces are Fr\'{e}chet, but it is also consistent
that none of them is. \ We use these spaces to exhibit (consistently) a
$\triangle_{2}^{1}$ ideal that does not satisfy the Category Dichotomy.
Finally, we prove that the Category Dichotomy holds for all co-analytic ideals.

\end{abstract}

\maketitle

\section{Introduction}

During the 2012 thematic program on Forcing and its Applications at the Fields
Institute, Juh\'{a}sz asked the following question:

\begin{problem}
[Juh\'{a}sz]Is there a countable, Fr\'{e}chet\footnote{Undefined concepts will
be reviewed in the following sections.} space with uncountable $\pi$-weight?
\label{Problema Juhasz}
\end{problem}

Although the consistency of existence of such spaces is easy to establish, no \textsf{ZFC
}result was known. The problem was motivated by the following result of
Hru\v{s}\'{a}k and Ramos:

\begin{theorem}
[H., Ramos \cite{MalykhinProblem}]It is consistent that every countable,
Fr\'{e}chet topological group is second countable.
\end{theorem}

Since for topological groups the notions of weight and $\pi$-weight coincide,
Juh\'{a}sz's question asks to what extent the algebraic structure is needed for
the consistent result above. The question is particularly interesting, as many
of the results in \cite{MalykhinProblem} do not use the algebraic structure at
all. The problem remained open until Dow provided a solution:

\begin{theorem}
[Dow \cite{PiWeightandFrechetProperty}]There exists a countable, Fr\'{e}chet,
zero dimensional space with $\pi$-weight at least $\mathfrak{b}.$
\label{Teorema Dow}
\end{theorem}

Not only is this result impressive, but its proof is highly ingenious and
illuminating. Given a $\mathfrak{b}$-scale $\mathcal{B}$, Dow defined a
topological space $\mathbb{D}(\mathcal{B})=(\omega^{<\omega},\mathbb{\tau
}_{\mathcal{B}})$ (which we call the \emph{Dow space of} $\mathcal{B}$) that
is zero dimensional and whose $\pi$-weight is exactly the bounding number
$\mathfrak{b}.$ Dow then proceeds to take a sequential modification to obtain
a topology $\sigma_{\mathcal{B}}$ extending $\tau_{\mathcal{B}}$ such that
$\mathbb{D}_{1}(\mathcal{B})=(\omega^{<\omega},\mathbb{\sigma}_{\mathcal{B}})$
is Fr\'{e}chet, zero dimensional and its $\pi$-weight is at least
$\mathfrak{b}$ (the exact $\pi$-weight is currently unknown). Evidently, the
advantage of $\mathbb{D}_{1}(\mathcal{B})$ over $\mathbb{D}(\mathcal{B})$ is
that it is Fr\'{e}chet. Nevertheless, the original space $\mathbb{D}%
(\mathcal{B})$ has certain advantages over $\mathbb{D}_{1}(\mathcal{B}).$
Mainly, the open sets of $\mathbb{D}(\mathcal{B})$ admit a very nice
combinatorial description and are intuitive to work with, but the same is no
longer true for $\mathbb{D}_{1}(\mathcal{B}).$ It is then natural to ask the
following question:\smallskip
\begin{center}%
\begin{tabular}
[c]{l}%
\textsf{Can the Dow space of a scale be a Fr\'{e}chet space?}%
\end{tabular}
\smallskip
\end{center}

In other words, we wanted to know if taking the sequential modification is
really needed. We will say that a $\mathfrak{b}$-scale is \emph{Fr\'{e}chet
}if its Dow space is Fr\'{e}chet. In this article, we will prove the following result:

\begin{theorem}
Both of the following statements are consistent (but not at the same time)
with \textsf{ZFC:}

\begin{enumerate}
\item Every $\mathfrak{b}$-scale is Fr\'{e}chet.

\item No $\mathfrak{b}$-scale is Fr\'{e}chet.
\end{enumerate}
\end{theorem}

 We finish the paper with an application to the
Kat\v{e}tov order on ideals and prove that the Category Dichotomy of
\cite{KatetovOrderonBorelIdeals} is true for all co-analytic ideals.

The structure of the paper is as follows. Section \ref{Secion notacion} covers
the necessary notation and definitions that will be used throughout the paper.
The next five sections provide the required background on cardinal invariants
of the continuum, filters and ideals, topology, and forcing that will be needed
in the article. The reader may skip these preliminary sections and return to
them as needed. The definition of the Dow space of a $\mathfrak{b}$-scale from
\cite{PiWeightandFrechetProperty} is reviewed in Section \ref{DowSpace}. In
Section \ref{Sec p=b} we prove that $\mathfrak{p=b}$ implies that all
$\mathfrak{b}$-scales are Fr\'{e}chet. In Section \ref{Sec una no} we produce
a model where there is a non-Fr\'{e}chet $\mathfrak{b}$-scale and in Section
\ref{Sec todas no} we build a model where no Fr\'{e}chet $\mathfrak{b}$-scales
exist. In Section \ref{CAT} we use results from the previous sections to find
a consistent example of a $\triangle_{2}^{1}$ ideal that does not satisfy the
Category Dichotomy and prove that this is the least possible complexity.
Section \ref{Sec pregunts} contains open questions that we do not know how to answer.
\newline

It is worth pointing out that countable, Fr\'{e}chet spaces with uncountable
$\pi$-weight have received a lot of attention recently, as can be seen from
the recently published papers \cite{NewExamplesSelectivelySeparable} and
\cite{SomeResultsPiWeightFrechet}.

\section{Notation \label{Secion notacion}}

Our notation is basically standard and follows \cite{Kunen} and \cite{Jech}.
Given $s,t\in\omega^{<\omega}$ by $s^{\frown}t$ we denote the
\emph{concatenation} of $s$ and $t.$ We write $s^{\frown}n$ instead of
$s^{\frown}\left(  n\right)  .$ If $F\subseteq\omega^{<\omega},$ define
$s^{\frown}F=\left\{  s^{\frown}z\mid z\in F\right\}  .$ We say that
$T\subseteq\omega^{<\omega}$ is a \emph{tree }if it is closed under taking
initial segments. For $s\in T,$ define \textsf{suc}$_{T}\left(  s\right)
=\left\{  n\mid s^{\frown}n\in T\right\}  $. We say $s$ \emph{is the stem of
}$T$ (denoted by $s=$ \textsf{st}$\left(  T\right)  $) if every node of $T$ is
comparable with $s$ and $s$ is maximal with this property. \ The domain of a
function $f$ is denoted by \textsf{dom}$\left(  f\right)  $ and its image by
\textsf{im}$\left(  f\right)  .$ By $f$ $;X\longrightarrow Y$ we mean that $f$
is a partial function from $X$ to $Y$ (i.e., \textsf{dom}$\left(  f\right)
\subseteq X$). \ The expression \textquotedblleft for almost
all\textquotedblright\ means \textquotedblleft for all except finitely
many\textquotedblright.

If $\kappa$ is a cardinal, define \textsf{H}$\left(  \kappa\right)  $ as the
set consisting of all sets whose transitive closure has size less than
$\kappa.$ We will work with elementary submodels of these sets. For an
introduction to this very important technique, we refer the reader to the
survey \cite{AlanSubmodelos}.

We will need the concepts of $F_{\sigma},$ $G_{\delta},$ Borel, analytic,
co-analytic and projective subsets of a Polish space, which can be consulted in
\cite{Kechris} or \cite{BorelSrivastava}.

\section{Preliminaries on cardinal invariants of the continuum}

\emph{Cardinal invariants of the continuum} play an important role in this
article. We now review some basic definitions that will be needed. To learn
more, see \cite{HandbookBlass} and \cite{TheIntegersandTopology}. For
$f,g\in\omega^{\omega}$, define $f\leq^{\ast}g$ if $f\left(  n\right)  \leq
g\left(  n\right)  $ holds for almost all $n\in\omega$. We say a family
$\mathcal{B}\subseteq\omega^{\omega}$ is \emph{unbounded }if $\mathcal{B}$ is
unbounded with respect to $\leq^{\ast}.$ A family $\mathcal{D}\subseteq
\omega^{\omega}$ is a \emph{dominating family }if for every $f\in
\omega^{\omega},$ there is $g\in\mathcal{D}$ such that $f\leq^{\ast}g.$ The
\emph{bounding number }$\mathfrak{b}$ is the size of the smallest unbounded
family and the \emph{dominating number }$\mathfrak{d}$ is the smallest size of
a dominating family. We say that $\mathcal{B=}\left\{  f_{\alpha}\mid\alpha
\in\mathfrak{b}\right\}  \subseteq\omega^{\omega}$ is a $\mathfrak{b}%
$\emph{-scale} if every $f_{\alpha}$ is an increasing function, $\mathcal{B}$
is unbounded and $f_{\alpha}<^{\ast}f_{\beta}$ whenever $\alpha<\beta.$ A
\emph{scale }is a dominating $\mathfrak{b}$-scale. It is not hard to see that
the existence of a scale is equivalent to $\mathfrak{b=d}.$

The order $\leq^{\ast}$ can be extended to partial functions. Given $f,g$
$;\omega\longrightarrow\omega$ define $f\leq^{\ast}g$ if $f\left(  n\right)
\leq g\left(  n\right)  $ holds for almost all $n$ in their common domain. It
is well-known that $\mathfrak{b}$-scales are not only unbounded with respect
to total functions, but also with respect to infinite partial functions. See \cite[Fact 3.4]{TheIntegersandTopology} for the proof of the following lemma:

\begin{lemma}
Let $\mathcal{B}\subseteq\omega^{\omega}$ be a $\mathfrak{b}$-scale and $g;\omega\longrightarrow\omega$ an infinite partial function. There is
$f\in\mathcal{B}$ such that $f\nleq^{\ast}g$. \label{escala para parciales}
\end{lemma}

Let $A$ and $B$ be two subsets of $\omega.$ Define $A\subseteq^{\ast}B$ ($A$
\emph{is an almost subset of} $B$) if $A\setminus B$ is finite. For
$\mathcal{H\subseteq}$ $\left[  \omega\right]  ^{\omega}$ and $A,B\subseteq
\omega,$ we say that $A\ $\emph{is a pseudo-intersection of }$\mathcal{H}$ if
it is almost contained in every element of $\mathcal{H}.$ A family
$\mathcal{P}\subseteq\left[  \omega\right]  ^{\omega}$ is \emph{centered }if
the intersection of finitely many of its elements is infinite. The
pseudo-intersection number $\mathfrak{p}$ is the least size of a centered
family without an infinite pseudo-intersection. We say $\mathcal{T}=\left\{
A_{\alpha}\mid\alpha<\kappa\right\}  $ is a \emph{pre-tower }if it is
$\subseteq^{\ast}$-decreasing and it is a \emph{tower }it has no infinite
pseudo-intersection. The \emph{tower number} $\mathfrak{t}$ is the least length
of a tower. By $\mathfrak{c}$ we denote the size of the real numbers. It is
easy to see that $\mathfrak{p\leq t\leq b\leq d\leq c}.$ Moreover, an
impressive theorem of Malliaris and Shelah establishes that $\mathfrak{p=t}$
(see \cite{cofinalityspectrum}). The following is essentially 
\cite[Theorem 4.2]{BaumgartnerDordalAdjoining}.

\begin{proposition}
[Baumgartner, Dordal \cite{BaumgartnerDordalAdjoining}] If there are no towers
of length $\mathfrak{b},$ then every $\mathfrak{b}$-scale is a scale (so
$\mathfrak{b=d}$). \label{BaumgartnerDordal}
\end{proposition}

\begin{proof}
Assume that there is a $\mathfrak{b}$-scale $\mathcal{B}=\left\{  f_{\alpha}%
\mid\alpha\in\mathfrak{b}\right\}  $ that is not a dominating family. Let
$g\in\omega^{\omega}$ be not dominated by any function in $\mathcal{B}.$ It
follows that the set $A_{\alpha}=\left\{  n\mid f_{\alpha}\left(  n\right)
<g\left(  n\right)  \right\}  $ is infinite for every $\alpha<\mathfrak{b}$
and $\mathcal{T}=\left\{  A_{\alpha}\mid\alpha\in\mathfrak{b}\right\}  $ is a
pretower. Since there are no towers of length $\mathfrak{b},$ there is
$X\in\left[  \omega\right]  ^{\omega}$ which is a pseudointersection of
$\mathcal{T}.$ We get $f_{\alpha}\upharpoonright X\leq^{\ast
}g\upharpoonright X$ for every $\alpha<\mathfrak{b}$, which contradicts Lemma
\ref{escala para parciales}.
\end{proof}

We will need the following result in Section \ref{Sec p=b}.

\begin{lemma}
Let $\mathcal{B}=\left\{  f_{\alpha}\mid\alpha\in\mathfrak{b}\right\}  $ be a
$\mathfrak{b}$-scale and $M$ an elementary submodel of \textsf{H}$\left(
\kappa\right)  $ (for some large enough regular cardinal $\kappa$) such that
$\mathcal{B}\in M,$ the size of $M$ is less than $\mathfrak{b}$ and
$\delta=M\cap\mathfrak{b\in b}.$ If $g; \omega\longrightarrow\omega$ is an
infinite partial function in $M,$ then $f_{\delta}\nleq^{\ast}g.$
\label{no acotado modelo}
\end{lemma}

\begin{proof}
By Lemma \ref{escala para parciales}, there is $\alpha<\mathfrak{b}$ such that
$f_{\alpha}\nleq^{\ast}g.$ Moreover, by elementarity, we may assume that
$\alpha<\delta.$ Since $f_{\alpha}\leq^{\ast}f_{\delta},$ the result follows.
\end{proof}

Let $M$ be a model of \textsf{ZFC }and $f\in\omega^{\omega}.$ We say that $f$
\emph{is dominating (unbounded) over} $M$ if for every $g\in M\cap
\omega^{\omega},$ it is the case that $g\leq^{\ast}f$ ($f\nleq^{\ast}g$).

\section{Preliminaries on filters and ideals}

We denote the power set of a set $X$ by $\mathcal{P}\left(  X\right)  $.
$\mathcal{I}\subseteq\mathcal{P}\left(  X\right)  $ is an \emph{ideal on }$X$
if $\emptyset\in\mathcal{I}$ and $X\notin\mathcal{I},$ for every $A,B\subseteq
X,$ if $A\in\mathcal{I}$ and $B\subseteq A$ then $B\in\mathcal{I}$ and if
$A,B\in\mathcal{I}$ then $A\cup B\in\mathcal{I}.$ A family
$\mathcal{F\subseteq}$ $\wp\left(  X\right)  $ is a called a \emph{filter on
}$X$ if $X\in\mathcal{F}$ and $\emptyset\notin\mathcal{F},$ for every
$A,B\subseteq X,$ if $A\in\mathcal{F}$ and $A\subseteq B$ then $B\in
\mathcal{F}$ and if $A,B\in\mathcal{F}$ then $A\cap B\in\mathcal{F}.$ Given a
family $\mathcal{B}$ of subsets of $X,$ we define $\mathcal{B}^{\ast}=\left\{
X\setminus B\mid B\in\mathcal{B}\right\}  .$ Note that if $\mathcal{F}$ is a
filter, then $\mathcal{F}^{\ast}$ is an ideal (called the \emph{dual ideal of
}$\mathcal{F}$) and if $\mathcal{I}$ is an ideal, then $\mathcal{I}^{\ast}$ is
a filter (called the dual filter of $\mathcal{I}$). Let $\mathcal{I}$ be an
ideal on $X$. The collection of $\mathcal{I}$\emph{-positive sets }is
$\mathcal{I}^{+}=\wp\left(  X\right)  \smallsetminus\mathcal{I}$ \emph{. }If
$\mathcal{F}$ is a filter, we define $\mathcal{F}^{+}$ $\mathcal{=}$ $\left(
\mathcal{F}^{\ast}\right)  ^{+}.$ It is easy to see that $\mathcal{F}^{+}$ is
the family of all sets that intersects every member of $\mathcal{F}.$ If
$A\in\mathcal{I}^{+}$ then \emph{the restriction of }$\mathcal{I}$ to $A$,
defined as $\mathcal{I\upharpoonright}A=\wp\left(  A\right)  \cap\mathcal{I}$,
is an ideal on $A.$ By $\mathcal{I}^{\perp}$ we denote the set of all sets
that have finite intersection with every member of $\mathcal{I}.$ It is easy
to see that $\mathcal{I}^{\perp}$ is an ideal. We review some properties of
ideals that will be needed in this article.

\begin{definition}
Let $\mathcal{I}$ be an ideal on $\omega$ (or any countable set).

\begin{enumerate}
\item $\mathcal{I}$ is \emph{tall }if for every $X\in\left[  \omega\right]
^{\omega}$ there is $Y\in\mathcal{I}$ such that $Y\cap X$ is infinite.

\item $\mathcal{I}$ is a \emph{Fr\'{e}chet ideal} if for every $A\in
\mathcal{I}^{+},$ there is $B\in\left[  A\right]  ^{\omega}\cap\mathcal{I}%
^{\perp}.$

\item $\mathcal{I}$ is $\omega$\emph{-hitting }if for every $\left\{
X_{n}\mid n\in\omega\right\}  \subseteq\left[  \omega\right]  ^{\omega}$ there
is $Y\in\mathcal{I}$ such that $Y\cap X_{n}$ is infinite for every $n\in
\omega.$

\item $\mathcal{I}$ is \emph{weakly selective }if for every $X\in
\mathcal{I}^{+}$ and $\mathcal{P}$ a partition of $X$ either $\mathcal{P\not\subseteq
I}$ or $\mathcal{P}$ has a (partial) selector in $\mathcal{I}%
^{+}$ (that is, a set that intersects each element of
$\mathcal{P}$ in no more than one point).
\end{enumerate}
\end{definition}

Let $\mathcal{F}$ be a filter on a set $W$. The filter $\mathcal{F}^{<\omega}$
is the filter on $[W]^{<\omega}\setminus\left\{  \emptyset\right\}  $ generated
by $\left\{  \left[  A\right]  ^{<\omega}\setminus\left\{  \emptyset\right\}
\mid A\in\mathcal{F}\right\}  .$ Note that a set $X$ is in $\left(
\mathcal{F}^{<\omega}\right)  ^{+}$ if and only if every $A\in\mathcal{F}$
contains an element of $X$.

\begin{definition}
Let $\mathcal{F}$ be a filter. We say that $\mathcal{F}$ is a \emph{FUF filter
}if for every $X\in\left(  \mathcal{F}^{<\omega}\right)  ^{+}$ there is
$Y\in\left[  X\right]  ^{\omega}$ such that every $F\in\mathcal{F}$ contains
almost all elements of $Y.$
\end{definition}

In this way, $\mathcal{F}$ is a \textsf{FUF }filter if and only if
$(\mathcal{F}^{<\omega})^{\ast}$ is a Fr\'{e}chet ideal. It is not hard to see
that every countably generated filter is a \textsf{FUF} filter. Gruenhage and
Szeptycki asked if there was a \textsf{ZFC }example of a \textsf{FUF} filter
on $\omega$ that is not countably generated. In \cite{CountableFrechetGroups}
it was proved that it is consistent that every \textsf{FUF }filter generated
by less than $\mathfrak{c}$ many elements is countably generated and the main
theorem of \cite{MalykhinProblem} implies that all \textsf{FUF }filters are
countably generated.

We now review the Kat\v{e}tov order, which will be needed in Section \ref{CAT}.

\begin{definition}
Let $X,Y$ be two sets, $\mathcal{I}$ an ideal on $X$, $\mathcal{J}$ an ideal
on $Y$ and $f:X\longrightarrow Y.$

\begin{enumerate}
\item $f$ is a \emph{Kat\v{e}tov function} from $\mathcal{I}$ to $\mathcal{J}$
if for every $A\subseteq Y$, we have that if $A\in\mathcal{J},$ then
$f^{-1}\left(  A\right)  \in\mathcal{I}.$

\item $\mathcal{J}\leq_{\text{\textsf{K}}}\mathcal{I}$ means that there is a
Kat\v{e}tov function from $\mathcal{I}$ to $\mathcal{J}.$
\end{enumerate}
\end{definition}

The \emph{eventually different} ideal $\mathcal{ED}$ is the ideal on
$\omega^{2}$ generated by the set of columns $\left\{  \left\{  n\right\}
\times\omega\mid n\in\omega\right\}  $ and the graphs of functions from
$\omega$ to $\omega.$ Fix $X$ a topological space and $N\subseteq X.$ We say
that $N$ is \emph{nowhere dense }if for every non-empty open set $U\subseteq
X,$ there is another open set $\emptyset\neq V\subseteq U$ such
that $V\cap N=\emptyset.$ By \textsf{nwd}$\left(  X\right)  $ we denote the
ideal of nowhere dense subsets of $X.$ By \textsf{nwd} we mean the ideal of
nowhere dense subsets of the rational numbers.

Readers interested in learning more about the Kat\v{e}tov order are encouraged
to see \cite{KatetovOrderonBorelIdeals}, \cite{OrderingMADFamiliesalaKatetov},
\cite{CombinatoricsofFiltersandIdeals}, \cite{StructuralKatetov},
\cite{ForcingIndestructibilityofMADFamilies}, \cite{KatetovOrderImply},
\cite{DicotomiaCat}, \cite{KatetovandKatetovBlassOrdersFsigmaIdeals},
\cite{SakaiKatetov}, \cite{KatetovMAD},
\cite{InvariancePropertiesofAlmostDisjointFamilies},
\cite{KwelaUltrafiltersKatetov}, \cite{KwelaExtendability} or
\cite{KatetovHindman} among others.

\section{Preliminaries on Topology}

We now review some basic topological concepts that will be used in the paper.
Let $X$ be a topological space\footnote{All spaces under discussion are
Hausdorff.} and $b\in X.$ Recall that $\left(  X,\tau\right)  $ is \emph{zero-dimensional }if it has a basis of clopen sets. For $A\subseteq X$ a
countable set, we say that $A$ \emph{converges to} $b$ (denoted by
$A\longrightarrow b$) if every open subset of $b$ almost contains $A.$ $X$ is
a \emph{Fr\'{e}chet space} if for every $a\in X$ and $Y\subseteq X$ such that
$a\in\overline{Y}\setminus Y,$ there is $A\in\left[  Y\right]  ^{\omega}$ that
converges to $a.$ Let $\mathcal{B}$ be a collection of non-empty open subsets of $X.$ We
say that $\mathcal{B}$ is a $\pi$\emph{-base }if every non-empty open subset of $X$
contains an element of $\mathcal{B}.$ The $\pi$\emph{-weight of} $X$ is the
smallest size of a $\pi$-base of $X.$ Moreover, $\mathcal{B}$ is a
\emph{local} $\pi$\emph{-base at }$b$ if every neighborhood of $b$ contains an
element of $\mathcal{B}.$ The $\pi$\emph{-character of} $b$ is the smallest
size of a \emph{local} $\pi$\emph{-base at }$b.$ We say that $X$ \emph{has
uncountable} $\pi$\emph{-character everywhere }if every point of $X$ has
uncountable $\pi$-character. Let $\varphi$ be a topological property. By
$X\models\varphi$ we mean that \textquotedblleft$X$ has property $\varphi
$\textquotedblright. We mainly use this notation when we are working with
several topological spaces with the same underlying set.

Let $X$ be a topological space and $a\in X.$ Denote by $\mathcal{N}_{X}\left(
a\right)  $ the neighborhood filter of $a$ and $\mathcal{I}_{X}\left(
a\right)  $ its dual ideal. Note that $\mathcal{I}_{X}\left(  a\right)
=\left\{  A\subseteq X\mid a\notin\overline{A}\right\}  .$ If there is no risk
of confusion, we simply write $\mathcal{N}\left(  a\right)  $ and
$\mathcal{I}\left(  a\right)  .$ Many topological properties at the point $a$
can be expressed as properties of its neighborhood filter and its dual ideal,
as shown in Table 1.

\begin{center}%
\begin{tabular}
[c]{|c|c|}\hline
\textbf{Topological } & \textbf{Combinatorial}\\
\textbf{property} & \textbf{translation}\\\hline
& \\
$a\in\overline{B}$ & $B\in\mathcal{N}\left(  a\right)  ^{+}$\\
& \\
$A\longrightarrow a$ & $A$ is a pseudointersection of $\mathcal{N}\left(
a\right)  $\\
& equivalently, $A\in\mathcal{I}\left(  a\right)  ^{\perp}$\\
& \\
$a$ is a Fr\'{e}chet point & $\mathcal{I}\left(  a\right)  $ is a Fr\'{e}chet
ideal\\
& \\
$a\in\overline{Y}$ but no sequence & $Y\in\mathcal{I}\left(  a\right)  ^{+}$
and $\mathcal{I}\left(  a\right)  \upharpoonright Y$ is\\
from $Y$ converges to $a$ & a tall ideal\\\hline
\end{tabular}

{\small Table 1. Topological properties and their translations}
\end{center}

We will need the following result in Section \ref{Sec todas no}, which is
Proposition 44 of \cite{DicotomiaCat}.

\begin{proposition}
Let $X$ be a countable Fr\'{e}chet space with no isolated points. The ideal
\textsf{nwd}$\left(  X\right)  $ is weakly selective. \label{nwd no arriba ED}
\end{proposition}

\section{Preliminaries on Forcing}

We review some preliminaries on forcing that will be needed. We assume the
reader is already familiar with the method of forcing as presented in
\cite{oldKunen}. Let $\mathcal{F}$ be a filter on $\omega.$ The \emph{Laver
forcing of }$\mathcal{F}$ (denoted by $\mathbb{L(\mathcal{F})}$) consists of
all trees $p\subseteq\omega^{<\omega}$ that have a stem and if $t\in T$
extends the stem$,$ then \textsf{suc}$_{T}\left(  s\right)  \in\mathcal{F}$.
Given $p,q\in\mathbb{L(\mathcal{F})},$ denote $p\leq q$ if $p\subseteq q.$ It is easy to
see that $\mathbb{L(\mathcal{F})}$ is a \textsf{c.c.c. }partial order. If
$p\in\mathbb{L(\mathcal{F})}$ and $s\in p,$ define $p_{s}=\left\{  t\in p\mid t\subseteq
s\vee s\subseteq t\right\}  .$ It is clear that $p_{s}\in\mathbb{L}\left(
\mathcal{F}\right)  ,$ it extends $p$ and in case \textsf{st}$\left(
p\right)  \subseteq s,$ we have that \textsf{st}$\left(  p_{s}\right)  =s.$
The \emph{Laver generic real }will be denoted by $l_{gen}\in\omega^{\omega}$
and is the only element that is a branch of every tree in the generic filter.
To learn more about this forcings, see
\cite{MathiasPrikryandLaverPrikryTypeForcing}.

By $\mathbb{D}$ we denote \emph{Hechler forcing}, which consists of all pairs
$\left(  s,f\right)  $ where $s\in\omega^{<\omega}$ and $f\in\omega^{\omega}.$
Define $\left(  s,f\right)  \leq\left(  t,g\right)  $ if $t\subseteq s,$
$g\left(  n\right)  \leq f\left(  n\right)  $ for every $n\geq\left\vert
t\right\vert $ and $s\left(  i\right)  \geq g\left(  i\right)  $ for every
$i\in$ \textsf{dom}$\left(  s\right)  \setminus$ \textsf{dom}$\left(
t\right)  .$ Hechler forcing is the standard \textsf{c.c.c. }forcing for
adding a dominating real. To learn about Hechler forcing, see \cite{Barty},
\cite{HechlerForcing} or \cite{HechlerPalumbo}.

A family $\mathcal{W}\subseteq\left[  \omega\right]  ^{\omega}$ is $\omega
$\emph{-hitting }if for every $\left\{  X_{n}\mid n\in\omega\right\}
\subseteq\left[  \omega\right]  ^{\omega},$ there is $W\in\mathcal{W}$ that
intersects every $X_{n}.$ We say $\mathbb{P}$ \emph{preserves} $\omega
$\emph{-hitting families }if every $\omega$-hitting family remain $\omega
$-hitting after forcing with $\mathbb{P}.$ It is not hard to see that a
forcing preserving $\omega$-hitting families can not fill towers. A more
refined version of the previous notion is the following:

\begin{definition}
Let $\mathbb{P}$ be a partial order. We say that $\mathbb{P}$ \emph{strongly
preserves} $\omega$\emph{-hitting families }if for every $\mathbb{P}$-name
$\dot{B}$ for an infinite subset of $\omega,$ there is $\left\{  B_{n}\mid
n\in\omega\right\}  \subseteq\left[  \omega\right]  ^{\omega}$ such that for
every $X\in\left[  \omega\right]  ^{\omega},$ if $\left\vert X\cap
B_{n}\right\vert =\omega$ for every $n\in\omega,$ then $\mathbb{P}$ forces
that $X\cap\dot{B}$ is infinite.
\end{definition}

In \cite{CountableFrechetGroups}, the reader can find a characterization of
the filters whose Laver forcing strongly preserves $\omega$-hitting families.
Moreover, it is also proved that $\mathbb{L}\left(  \mathcal{F}\right)  $
preserves $\omega$-hitting families if and only if $\mathbb{L}\left(
\mathcal{F}\right)  $ strongly preserves $\omega$-hitting families. It is also
known that Hechler forcing strongly preserves $\omega$-hitting families. The
following iteration theorem can be found in \cite{CountableFrechetGroups}.

\begin{theorem}
[Brendle, H. \cite{CountableFrechetGroups}]The finite support iteration of
\textsf{c.c.c. }forcings that strongly preserve $\omega$-hitting families,
strongly preserves $\omega$-hitting families.
\end{theorem}

Let $\mathbb{P}$ be a partial order. We say that $C\subseteq\mathbb{P}$ is
\emph{centered }if any finitely many elements of $C$ have a lower bound in
$\mathbb{P}.$ Recall that $\mathbb{P}$ is $\sigma$\emph{-centered} if it is
the union of countably many of its centered sets. On the other hand, we say
that $\mathbb{P}$ is $\sigma$\emph{-filtered }if it is the union of countably
many filters. Evidently, every $\sigma$-filtered forcing is $\sigma$-centered.
Juh\'{a}sz and Kunen proved that the converse is not true (see
\cite{OnCenteredPosets}). Nevertheless, we have the following:

\begin{proposition}
Let $\mathbb{P}$ be a partial order and $\mathbb{B}$ its Boolean completion.
The following are equivalent: \label{Prop Boolean compl}

\begin{enumerate}
\item $\mathbb{B}$ is $\sigma$-centered.

\item $\mathbb{B}$ is $\sigma$-filtered.

\item $\mathbb{P}$ is $\sigma$-centered.
\end{enumerate}
\end{proposition}

We will need the following well-known preservation result (see
\cite{Tallsigmacentered} and \cite{BorelandDualBorel}).

\begin{proposition}
Let $\gamma<\mathfrak{c}^{+}$ and $\langle\mathbb{P}_{\alpha},\mathbb{\dot{Q}%
}_{\alpha}\mid\alpha<\gamma\rangle$ be a finite support iteration of $\sigma
$-centered forcings. $\mathbb{P}_{\gamma}$ is $\sigma$-centered.
\label{Prop iteracion centered}
\end{proposition}

A remarkable theorem of Bell is the following:

\begin{theorem}
[Bell \cite{TeoremadeBell}]Let $\kappa$ be a cardinal. The following are
equivalent: \label{Teorema de Bell}

\begin{enumerate}
\item $\kappa<\mathfrak{p}.$

\item For every $\sigma$-centered forcing $\mathbb{P}$ and $\left\{
D_{\alpha}\mid\alpha<\kappa\right\}  $ a collection of dense subsets of
$\mathbb{P},$ there is a filter $G\subseteq\mathbb{P}$ such that $G\cap
D_{\alpha}\neq\emptyset$ for every $\alpha<\kappa.$
\end{enumerate}
\end{theorem}

\section{Preliminaries on Forcing and topology \label{hitting and sealing}}

In this section we review the method for destroying the Fr\'{e}chet property
at a point, as developed in \cite{MalykhinProblem} (which was based on the
techniques from \cite{CountableFrechetGroups}). Let $X$ be a topological space
and $a\in X.$ According to Table 1, to destroy the Fr\'{e}chet property at $a$
by a forcing $\mathbb{P},$ we must add a set $\dot{A}$ such that
$\mathcal{I}\left(  a\right)  \upharpoonright\dot{A}$ is a tall ideal. But
this is not enough, since the tallness of an ideal may not be preserved under
forcing iterations. The solution is to ensure that $\mathcal{I}\left(
a\right)  \upharpoonright\dot{A}$ is not only a tall ideal but an $\omega
$-hitting one, for which we can prove preservation theorems under forcing.

\begin{definition}
Let $\mathcal{I}$ be an ideal on a countable set, $\mathbb{P}$ a forcing
notion, and $\dot{A}$ a $\mathbb{P}$-name. We say that $\mathbb{P}$
\emph{seals} $\mathcal{I}$ \emph{with }$\dot{A}$ if $\mathbb{P}$ forces that
$\dot{A}\in\mathcal{I}^{+}$ and $\mathcal{I}\upharpoonright\dot{A}$ is
$\omega$-hitting.
\end{definition}

Let $X$ be a countable space. It is easy to see that $\mathbb{L}%
($\textsf{nwd(}$X$)$^{\ast})$ forces $\dot{A}_{\text{\textsf{gen}}}$ to be a
dense subset of $X.$ The following is Proposition 5.2 of
\cite{MalykhinProblem}:

\begin{proposition}
[H., Ramos \cite{MalykhinProblem}]Let $X$ be a countable space with no
isolated points and $a\in X:$ \label{Prop Michael Ariet}

\begin{enumerate}
\item If $a$ has uncountable $\pi$-weight, then $\mathbb{L}($\textsf{nwd(}%
$X$)$^{\ast})$ seals $\mathcal{I}_{X}\left(  a\right)  $ via $\dot
{A}_{\text{\textsf{gen}}}.$

\item If $X$ is Fr\'{e}chet, then $\mathbb{L}($\textsf{nwd(}$X$)$^{\ast})$
strongly preserves $\omega$-hitting families.
\end{enumerate}
\end{proposition}

In this way, if $X$ and $a$ are as in the above proposition, then
$\mathbb{L}($\textsf{nwd(}$X$)$^{\ast})$ forces $\dot{A}_{\text{\textsf{gen}}%
}$ to be a dense subset of $X,$ yet it does not contain sequences converging
to $a.$ This property will be preserved under any further forcing extension
that preserves $\omega$-hitting families.

We would like to point out that the method of \cite{MalykhinProblem} has been
greatly refined and expanded by Shibakov and the third author (see \cite{IIA},
\cite{IIABeyond} and \cite{ConvergentSequencesinGroups}).

\section{The Dow space of a $\mathfrak{b}$-scale \label{DowSpace}}

We now review the construction of the Dow space from
\cite{PiWeightandFrechetProperty}. We expect that the study of the topological
properties of Dow spaces will be useful for investigating the combinatorial
properties of $\mathfrak{b}$-scales, much as the study of the
Mr\'{o}wka-Isbell spaces resides in understanding the combinatorics of almost
disjoint families (see \cite{TheIntegersandTopology},
\cite{AlmostDisjointFamiliesandTopology} and
\cite{TopologyofMrowkaIsbellSpaces}). Instead of working with $\mathfrak{b}%
$-scales on $\omega$, we find it more convenient to work on $\mathfrak{b}%
$-scales consisting of functions from $\omega^{<\omega}$ to $\omega.$ To this
end, we first adapt the relevant definitions to our setting.

For convenience, given $m\in\omega,$ we will denote $\triangle_{m}=m^{\leq
m}.$ Let $f,g:\omega^{<\omega}\longrightarrow\omega.$ We say that $f$ \emph{is
increasing }if for every $s,t\in\omega^{<\omega},$ if $s$ is a proper initial
segment of $t,$ then $f\left(  s\right)  <f\left(  t\right)  $ and for every
$n,m\in\omega,$ if $n<m,$ then $f(s^{\frown}n)<f(s^{\frown}m).$ As expected,
define $f<^{\ast}g$ if $f\left(  s\right)  <g\left(  s\right)  $ holds for
almost all $s\in\omega^{<\omega}.$ Moreover, define $f<_{m}g$ if $f\left(
s\right)  <g\left(  s\right)  $ holds for all $s\notin\triangle_{m}.$ It
follows that $f<^{\ast}g$ if and only if there is $m\in\omega$ such that
$f<_{m}g.$

\begin{definition}
Let $\mathcal{B}$ be a family of functions from $\omega^{<\omega}$ to
$\omega.$

\begin{enumerate}
\item $\mathcal{B}$ is a \emph{weak} $\mathfrak{b}$\emph{-scale }if the
following conditions hold:

\begin{enumerate}
\item $\mathcal{B}$ consists of increasing functions.

\item $\mathcal{B}$ is unbounded (there is no function $g$ such that
$f\leq^{\ast}g$ for every $f\in\mathcal{B}$).

\item $\mathcal{B}$ is well-ordered by $\leq^{\ast}.$

\item For every $n\in\omega,$ the set $\{f\in\mathcal{B}\mid n<f\left(
\emptyset\right)  \}$ is cofinal in $\mathcal{B}.$
\end{enumerate}

\item $\mathcal{B}$ is a $\mathfrak{b}$\emph{-scale }if its order type (with
respect to $\leq^{\ast}$) is $\mathfrak{b}.$
\end{enumerate}
\end{definition}

Although our main interest is $\mathfrak{b}$-scales, it is convenient to also
consider weak $\mathfrak{b}$-scales. We will now define the sets that will be
part of a subbase in a Dow space.

\begin{definition}
Let $f:\omega^{<\omega}\longrightarrow\omega.$ Define the tree $U\left(
f\right)  \subseteq\omega^{<\omega}$ recursively as follows:

\begin{enumerate}
\item $\emptyset\in U\left(  f\right)  .$

\item If $s\in U\left(  f\right)  ,$ then \textsf{suc}$_{U\left(  f\right)
}\left(  s\right)  =\omega\setminus\left\{  f\left(  s\right)  \right\}  .$
\end{enumerate}
\end{definition}

In this way, $U\left(  f\right)  $ is a very wide tree, since every node
branches into all elements of $\omega$ except one. The following is
 easy, but  worth pointing out.

\begin{lemma}
Let $f:\omega^{<\omega}\longrightarrow\omega$ and $s\in\omega^{<\omega}.$ If
$s\notin U\left(  f\right)  ,$ then there is $i<\left\vert s\right\vert $ such
that $f\left(  s\upharpoonright i\right)  =s\left(  i\right)  .$
\label{Lema hoyito}
\end{lemma}

We now have the following easy lemma:

\begin{lemma}
Let $\mathcal{B}$ be a weak $\mathfrak{b}$-scale, $s\in\omega^{<\omega}$ and
$f\in\mathcal{B}.$
\label{Lema prop basicas}

\begin{enumerate}
\item If $s\in U\left(  f\right)  ,$ then $s^{\frown}\triangle_{f\left(
s\right)  }\subseteq U\left(  f\right)  .$

\item The set $\{g\in\mathcal{B}\mid s\in U\left(  g\right)  \}$ is cofinal in
$\mathcal{B}.$
\end{enumerate}
\end{lemma}

\begin{proof}
The first point follows since $f$ is an increasing function. For the second
point, choose $n\in\omega$ such that $s\in\triangle_{n}.$ Since $\{f\in
\mathcal{B}\mid n<f\left(  \emptyset\right)  \}$ is cofinal in $\mathcal{B},$
we get the desired conclusion.
\end{proof}

Let $s\in\omega^{<\omega}.$ In this paper, we denote $\left\langle
s\right\rangle =\left\{  t\in\omega^{<\omega}\mid s\subseteq t\right\}  .$
This set should not be confused with $\left\{  f\in\omega^{\omega}\mid
s\subseteq f\right\}  ,$ which is also denoted by $\left\langle s\right\rangle
$ in the literature. Sets of the form $\left\langle s\right\rangle $ will be
often referred as \emph{cones}, while a \emph{cocone }is a set of the form
$\left\langle s\right\rangle ^{\text{\textsf{c}}}=\omega^{<\omega}%
\setminus\left\langle s\right\rangle .$ A \emph{non-trivial cocone} is simply
a non-empty cocone. Define\footnote{In \cite{PiWeightandFrechetProperty} our
$s^{\uparrow}$ is denoted by $s^{\downarrow}$. We reverse this notation, as we
picture our trees as growing downward.} $s^{\uparrow}=\left\{  t\in
\omega^{<\omega}\mid t\subseteq s\right\}  $ and for a set $A\subseteq
\omega^{<\omega},$ denote $A^{\uparrow}=\bigcup\left\{  t^{\uparrow}\mid t\in
A\right\}  .$ We can now define the Dow space of a weak $\mathfrak{b}$-scale.

\begin{definition}
Let $\mathcal{B}$ be a weak $\mathfrak{b}$-scale. The \emph{Dow space of
}$\mathcal{B}$ is the space $\mathbb{D}(\mathcal{B})=(\omega^{<\omega}%
,\tau_{\mathcal{B}})$ where $\mathcal{\tau}_{\mathcal{B}}$ is the topology
generated by the subbase consisting of the following sets:

\begin{enumerate}
\item $\left\langle s\right\rangle ,$ $\left\langle s\right\rangle
^{\text{\textsf{c}}}$ for $s\in\omega^{<\omega}$.

\item $U\left(  f\right)  $ for $f\in\mathcal{B}.$
\end{enumerate}
\end{definition}

In \cite{PiWeightandFrechetProperty} the topology is not explicitly defined
from the $\mathfrak{b}$-scale (as in our presentation), but it is instead
constructed recursively. Nevertheless, our presentation falls under the scope
of \cite{PiWeightandFrechetProperty} because the family $\{U\left(  f\right)
\mid f\in\mathcal{B}\}$ satisfies Lemma 3.1 of that paper. The following
notion was introduced in \cite{PiWeightandFrechetProperty}:

\begin{definition}
We say a topological space $(\omega^{<\omega},\tau)$ is $\uparrow
$\emph{-sequential }if for every $s\in\omega^{<\omega},$ the following
conditions hold:

\begin{enumerate}
\item $\left\langle s\right\rangle $ is an open set$.$

\item $\left\langle s^{\frown}n\right\rangle _{n\in\omega}$ converges to $s.$

\item If $A\subseteq\omega^{<\omega}$ converges to $s,$ then $A^{\uparrow}$
also converges to $s.$
\end{enumerate}
\end{definition}

Denote by $A^{\left(  1\right)  }$ the set of all convergence points of
sequences contained in $A.$ In \cite{PiWeightandFrechetProperty} Dow proved
the following:

\begin{theorem}
[Dow]Let $\mathcal{B}$ be a $\mathfrak{b}$-scale. \ \label{Teorema prop Dow}

\begin{enumerate}
\item $\mathbb{D}\left(  \mathcal{B}\right)  $ is zero dimensional.

\item $\mathbb{D}\left(  \mathcal{B}\right)  $ is $\uparrow$-sequential.

\item The $\pi$-character of every point in$\ \mathbb{D}\left(  \mathcal{B}%
\right)  $ is $\mathfrak{b}.$

\item If $A\subseteq\omega^{<\omega},$ then $\left(  A^{\left(  1\right)
}\right)  ^{\left(  1\right)  }=A^{\left(  1\right)  }.$

\item $\mathbb{D}\left(  \mathcal{B}\right)  $ has no isolated points.

\item Every $\uparrow$-sequential topology extending $\tau_{\mathcal{B}}$ has
$\pi$-character at least $\mathfrak{b}.$
\end{enumerate}
\end{theorem}

Let $U\subseteq\omega^{<\omega}$ and $n\in\omega.$ Define $U_{>n}=\left\{
s\in U\mid s=\emptyset\vee s\left(  0\right)  >n\right\}  .$ We have the following:

\begin{lemma}
Let $\mathcal{B}$ be a weak $\mathfrak{b}$-scale. The family:\medskip

\hfill%
\begin{tabular}
[c]{l}%
$\left\{  \left(  U(f_{1})\cap...\cap U\left(  f_{n}\right)  \right)
_{>m}\mid f_{1},...,f_{n}\in\mathcal{B}\wedge m\in\omega\right\}  $%
\end{tabular}
\hfill\qquad\medskip

is a local base of $\emptyset.$
\end{lemma}

More constructions of countable Fr\'{e}chet spaces with uncountable $\pi
$-weight can be found in \cite{NewExamplesSelectivelySeparable}.
\section{All $\mathfrak{b}$-scales may be Fr\'{e}chet \label{Sec p=b}}

In this section we will prove that it is consistent that every $\mathfrak{b}%
$-scale is Fr\'{e}chet. In fact, we will prove that the equality
$\mathfrak{p=b}$ implies that the Dow space of every $\mathfrak{b}$-scale
satisfies a strong form of the Fr\'{e}chet property. We recall the following definition:

\begin{definition}
Let $X$ be a topological space. We say that $X$ is \emph{Fr\'{e}chet-Urysohn for
finite sets} if for every $a\in X,$ its neighborhood filter $\mathcal{N}%
\left(  a\right)  $ is a \textsf{FUF }filter.
\end{definition}

This class of spaces has been studied in \cite{Maybe1Countable}, \cite{Olga},
\cite{PaulyGary}, \cite{FUF2}, \cite{CountableFrechetGroups} and
\cite{MoreonFUF} among many others. It is not hard to see that every space
that is Fr\'{e}chet-Urysohn for finite sets is also Fr\'{e}chet.

\begin{theorem}[$\mathfrak{p=b}$] \label{AllAreFrechet}
The Dow space of every $\mathfrak{b}$-scale is
Fr\'{e}chet-Urysohn for finite sets (and in particular, it is a Fr\'{e}chet
space). \label{Teo p=b}
\end{theorem}

\begin{proof}
Let $\mathcal{B}=\{f_{\alpha}\mid\alpha<\mathfrak{b\}}$ be a $\mathfrak{b}%
$-scale. For simplicity, we will prove that the neighborhood filter of
$\emptyset$ is a \textsf{FUF }filter. The argument for an arbitrary
$s\in\omega^{<\omega}$ is essentially the same, only requiring more notation.
For ease of notation, let $\mathcal{F=N}\left(  \emptyset\right)  $ and
$U_{\alpha}=U\left(  f_{\alpha}\right)  $ for $\alpha<\mathfrak{b}.$ Let
$X\in(\mathcal{F}^{<\omega})^{+}.$ We first find $M$ an elementary submodel of
\textsf{H}$\left(  \kappa\right)  $ (for some large enough regular cardinal
$\kappa$) such that $X,\mathcal{B}\in M,$ the size of $M$ is less than
$\mathfrak{b}$ and $\delta=M\cap\mathfrak{b\in b}.$ Denote by $\mathcal{U}$
the set of all $\bigcap\limits_{\alpha\in F}U_{\alpha}$ for $F\in\left[
\delta\right]  ^{<\omega}.$

Since every non-trivial co-cone is a neighborhood of $\emptyset,$ if follows
that for every $U\in\mathcal{N}\left(  \emptyset\right)  $ and $n\in\omega,$
there is $a\in X$ such that $a\subseteq U$ and $n<s\left(  0\right)  $ for
every $s\in a\setminus\left\{  \emptyset\right\}  .$ In this way, for every
$U\in\mathcal{U},$ we can define the function $g_{U}\in\omega^{\omega}$ such
that for every $n\in\omega,$ it is the case that $g_{U}\left(  n\right)  $ is
the least natural number for which there is $a\in X$ with the following properties:

\begin{enumerate}
\item $a\subseteq U.$

\item $n<s\left(  0\right)  $ for every $s\in a\setminus\left\{
\emptyset\right\}  .$

\item If $s\in a,$ then $s\in\triangle_{g_{U}\left(  n\right)  }.$
\end{enumerate}

Note that $g_{U}\in M$ for every $U\in\mathcal{U}.$ Lemma
\ref{no acotado modelo} implies that for every $V\in\mathcal{U},$ there are
infinitely many $n\in\omega$ such that $g_{U}\left(  n\right)  <f_{\delta
}\left(  \left(  n\right)  \right)  $ (of course, $f_{\delta}\left(  \left(
n\right)  \right)  $ is the result of applying $f_{\delta}$ to $\left(
n\right)  \in\omega^{<\omega}$). We now define $\mathbb{P}$ as the set of all
$p=\left(  F,h,H\right)  $ with the following properties:

\begin{enumerate}
\item $F\in\left[  \omega\right]  ^{<\omega},$ $h:F\longrightarrow X$ and
$H\in\left[  \delta\right]  ^{<\omega}.$

\item For every $n\in F$ and $s\in h\left(  n\right)  \setminus\left\{
\emptyset\right\}  ,$ the following conditions hold:

\begin{enumerate}
\item $n<s\left(  0\right)  .$

\item $s\in\triangle_{f_{\delta}\left(  \left(  n\right)  \right)  }.$
\end{enumerate}
\end{enumerate}

Let $p,q\in\mathbb{P}.$ Define $p\leq q$ in case the following conditions are met:

\begin{enumerate}
\item $F_{q}\subseteq F_{p},$ $h_{q}\subseteq h_{p}$ and $H_{q}\subseteq
H_{p}.$

\item For every $n\in F_{p}\setminus F_{q}$ and $\alpha\in H_{q},$ we have
that $h_{p}\left(  n\right)  \subseteq U_{\alpha}.$
\end{enumerate}

It is easy to see that $\mathbb{P}$ is a $\sigma$-centered forcing (any finite
set of conditions that share the same first two coordinates are compatible).
The following claim is also not hard to prove:

\begin{enumerate}
\item For every $n\in\omega,$ the set $D_{n}=\{p\in\mathbb{P\mid}%
F_{p}\nsubseteq n\mathbb{\}}$ is dense.

\item For every $\alpha<\delta,$ the set $E_{\alpha}=\{p\in\mathbb{P\mid
}\alpha\in H_{p}\mathbb{\}}$ is dense.
\end{enumerate}

Since $\delta<\mathfrak{b=p},$ by Theorem \ref{Teorema de Bell}, we can find a
filter $G\subseteq\mathbb{P}$ that intersects all of the dense sets described
above. Define $F_{G}=\bigcup\limits_{p\in G}F_{p},$ $h_{G}=\bigcup
\limits_{p\in G}h_{p}$ (note that $h_{G}:F_{G}\longrightarrow\omega$) and
$Y\subseteq X$ the image of $h_{G}.$ We claim that $Y$ is as desired. We need
to prove that every $U\in\mathcal{N}\left(  \emptyset\right)  $ contains
almost every element of $Y.$ Note that we may assume that $U$ is a sub-basic
set. Moreover, if $U$ is a non-trivial co-cone, the conclusion is
straightforward, since the first coordinate of every element of $h_{G}\left(
n\right)  \setminus\left\{  \emptyset\right\}  $ is larger than $n.$

It remains to prove that if $\alpha<\mathfrak{b},$ then $a\subseteq U_{\alpha
}$ for almost all $a\in Y.$ This is clearly the case if $\alpha<\delta$ (since
$G\cap E_{\alpha}\neq\emptyset$ and $G$ is a filter), so we may assume that
$\delta\leq\alpha.$ Let $n\in F_{G}$ such that $f_{\alpha}\left(
\emptyset\right)  <n$ and $f_{\delta}\left(  \left(  m\right)  \right)  \leq
f_{\alpha}\left(  \left(  m\right)  \right)  $ for every $m\geq n$ (almost
every element of $F_{G}$ satisfies this conditions). We need to prove that
$h_{G}\left(  n\right)  \subseteq U_{\alpha}.$ Write $h_{G}\left(  n\right)
\setminus\left\{  \emptyset\right\}  =\left\{  s_{0},...,s_{l}\right\}  .$ We
know that $n<s_{0}\left(  0\right)  ,...,s_{l}\left(  0\right)  $ and
$s_{0},...,s_{l}\in\triangle_{f_{\delta}\left(  \left(  n\right)  \right)  }.$
Moreover, $\left(  s_{0}\left(  0\right)  \right)  ,...,\left(  s_{l}\left(
0\right)  \right)  \in U_{\alpha},$ since these values are above $f_{\alpha
}\left(  \emptyset\right)  .$ For each $i\leq l,$ we have that $s_{i}$
$\in\triangle_{f_{\delta}\left(  \left(  n\right)  \right)  }\subseteq
\triangle_{f_{\alpha}\left(  \left(  n\right)  \right)  }.$ It follows by
Lemma \ref{Lema prop basicas} that each $s_{i}$ is in $U_{\alpha}.$
\end{proof}

In \cite{NewExamplesSelectivelySeparable} Dow and Pecoraro proved that there
is a countable, zero dimensional space that is not \textsf{H}-separable and
has $\pi$-weight $\mathfrak{b.}$ In \cite{SubsetsofFrechet} Nyikos proved that $\mathfrak{p=b}$ implies that
there is an uncountably generated \textsf{FUF }filter. Our result provides
another proof of this result.

In response to Theorem \ref{Teorema Dow}, Moore posed the following problem:

\begin{problem}
[Moore]Is there a countable, Fr\'{e}chet, zero dimensional space with $\pi
$-weight exactly $\mathfrak{b}$?
\end{problem}

The Dow space of a Fr\'{e}chet $\mathfrak{b}$-scale provides such an example.
However, while Dow spaces have $\pi$-weight $\mathfrak{b},$ the $\pi$-weight
of their sequential modification remains unknown. Of course, the above problem
would have a positive solution if there existed a Fr\'{e}chet $\mathfrak{b}%
$-scale, but we will see in a later section that this is consistently false.
Nevertheless, we have the following result, which was independently proved by Dow and Pecoraro in \cite{NewExamplesSelectivelySeparable}. 

\begin{proposition}
If $\mathfrak{c}\leq\omega_{2},$ then there is a countable, Fr\'{e}chet, zero-dimensional space of $\pi$-weight exactly $\mathfrak{b}.$
\end{proposition}

\begin{proof}
If $\mathfrak{p=b},$ then every $\mathfrak{b}$-scale is Fr\'{e}chet by Theorem
\ref{Teo p=b}. Otherwise, we $\mathfrak{p<b=c},$ so there is such space by Theorem \ref{Teorema Dow}.
\end{proof}

In \cite{SomeResultsPiWeightFrechet}\ Dow investigated the possible $\pi
$-weights of countable, regular, Fr\'{e}chet spaces. He proved that in the
Miller model every such space has $\pi$-weight at most $\omega_{1}.$ On the
other hand, after adding $\kappa$ many random reals, for any infinite cardinal
$\lambda\leq\kappa,$ there exists a countable, regular, Fr\'{e}chet space with
$\pi$-weight exactly $\lambda.$ An open question remains whether there are
(consistently) uncountable cardinals $\lambda<\mu<\kappa$ such that both
$\lambda$ and $\kappa$ are realized as the $\pi$-weight of a countable,
regular, Fr\'{e}chet space, while $\mu$ is not.
\section{There may be a non-Fr\'{e}chet $\mathfrak{b}$-scale
\label{Sec una no}}

In this section, we prove that it is consistent that there is a $\mathfrak{b}$-scale whose Dow space is not Fr\'{e}chet. This result will be strengthened
in the next section, where we show that it is consistent that no
$\mathfrak{b}$-scale is Fr\'{e}chet. Although the proof in this section
motivates some of the ideas used later, the argument here is not a special
case of the one in the next section. More importantly, we conjecture that
there will be Fr\'{e}chet $\mathfrak{b}$-scales in the model constructed in
this section, which would establish the consistency of the simultaneous
existence of both a Fr\'{e}chet and a non-Fr\'{e}chet $\mathfrak{b}$-scale,
which is unknown at the moment. We recommend that the reader consult Section
\ref{hitting and sealing} as this section and the next one, use definitions and
results from there.

For ease of writing, if $\mathcal{D}$ is a weak $\mathfrak{b}$-scale, we will
write $\mathbb{L(\mathcal{D})}$ instead of $\mathbb{L}$(\textsf{nwd(}%
$\mathbb{D(}\mathcal{D})$)$^{\ast}$). By Theorems \ref{Teorema prop Dow},
\ref{Teo p=b} and Proposition \ref{Prop Michael Ariet}, we get the following:

\begin{lemma}
\label{Lema Laver escala}

\begin{enumerate}
\item Let $\mathcal{D}$ be a weak $\mathfrak{b}$-scale such that
$\mathbb{D(\mathcal{D})}$ is Fr\'{e}chet. $\mathbb{L(\mathcal{D})}$ strongly
preserves $\omega$-hitting families and forces that $\dot{A}%
_{\text{\textsf{gen}}}$ is a dense set, yet it does not contain sequences
convergent to $\emptyset.$

\item (\textsf{CH}) If $\mathcal{B}$ is a $\mathfrak{b}$-scale, then
$\mathbb{L(\mathcal{B})}$ strongly preserves $\omega$-hitting families and
forces that $\dot{A}_{\text{\textsf{gen}}}$ is a dense set, yet it does not
contain sequences converging to $\emptyset.$
\end{enumerate}
\end{lemma}

We can now prove the main result of this section.

\begin{theorem}
It is consistent that there is a $\mathfrak{b}$-scale whose Dow space is not
Fr\'{e}chet. \label{Una no Frechet}
\end{theorem}

\begin{proof}
We start with a model of \textsf{CH. }Let $\mathcal{D}=\left\{  f_{\alpha}%
\mid\alpha\in\omega_{1}\right\}  $ be a $\mathfrak{b}$-scale. Define
$\mathbb{P=L(\mathcal{D})\ast\dot{D}}_{\omega_{2}},$ where $\mathbb{D}%
_{\omega_{2}}$ denotes the finite-support iteration of Hechler forcing of
length $\omega_{2}$. Let $G\subseteq\mathbb{P}$ be a generic filter. We will
prove that in $V\left[  G\right]  $ there exists a non-Fr\'{e}chet
$\mathfrak{b}$-scale.

We go to $V\left[  G\right]  .$ Let $A=A_{\text{\textsf{gen}}}\setminus
\left\{  \emptyset\right\}  ,\ $where $A_{\text{\textsf{gen}}}$ is the range
of the generic real of $\mathbb{L}\left(  \mathcal{D}\right)  .$ By Lemma
\ref{Lema Laver escala} and the strong preservation of $\omega$-hitting of
Hechler forcing, we know that in $\mathbb{D(\mathcal{D})}$ the set $A$ is dense
and does not contain converging sequences to $\emptyset.$ Of course,
$\mathcal{D}$ is no longer a $\mathfrak{b}$-scale (it is bounded). Let
$\mathcal{B=}\left\{  f_{\alpha}\mid\alpha\in\omega_{2}\right\}  $ be a
$\mathfrak{b}$-scale extending $\mathcal{D}$ such that $f_{\omega_{1}}$ is
dominating over $V\left[  A\right]  .$ We will prove that $\mathcal{B}$ is not Fr\'{e}chet.

\begin{claim}

\begin{enumerate}
\item $\mathbb{D(\mathcal{B})\models}$ $\emptyset\in\overline{A}.$

\item $\mathbb{D}\left(  \mathcal{B}\right)  \models A$ does not contain a
convergent sequence to $\emptyset.$
\end{enumerate}
\end{claim}

The second point is easy. Since $A$ does not contain a convergent sequence to
$\emptyset$ in $\mathbb{D(\mathcal{D})},$ then it cannot contain one in
$\mathbb{D(\mathcal{B})}$ (every open set in $\mathbb{D}\left(  \mathcal{D}%
\right)  $ is open in $\mathbb{D(\mathcal{B})}$)$.$ We now prove the first
point. The argument is similar to the proof of Theorem \ref{Teo p=b}.

For convenience, denote $U_{\alpha}=U(f_{\alpha})$ for $\alpha<\omega_{2}.$
Let $\mathcal{U}$ be the set of all $\bigcap\limits_{\alpha\in F}U_{\alpha}$
for $F\in\left[  \omega_{1}\right]  ^{<\omega}.$ Evidently, every element of
$\mathcal{U}$ is an open set in $\mathbb{D}\left(  \mathcal{D}\right)  .$

\textbf{Subclaim. }Let $W\in\mathcal{U}.$ The set $\left\langle \left(
n\right)  \right\rangle \cap W\cap A$ is not empty for almost all $n\in
\omega.$

Indeed, if $n$ is such that $\left\langle \left(  n\right)  \right\rangle \cap
W\neq\emptyset$ (which are almost all $n\in\omega$), since $\mathbb{D}\left(
\mathcal{D}\right)  \models A$ is dense, it follows that $\left\langle \left(
n\right)  \right\rangle \cap W\cap A\neq\emptyset.$ This finishes the proof of
the subclaim.

Given $W\in\mathcal{U},$ we can define $g_{W}\in\omega^{\omega}$ such that for
every $n\in\omega,$ the following holds:

\begin{enumerate}
\item $g_{W}\left(  n\right)  =0$ if $\left\langle \left(  n\right)
\right\rangle \cap W\cap A=\emptyset.$

\item If $\left\langle \left(  n\right)  \right\rangle \cap W\cap
A\neq\emptyset,$ let $g_{W}\left(  n\right)  $ be the least natural number
such that $\left(  n\right)  ^{\frown}\triangle_{g_{W}\left(  n\right)  }\cap
W\cap A\neq\emptyset.$
\end{enumerate}

Note that if $W\in\mathcal{U},$ then $g_{W}\in V\left[  A\right]  .$ It
follows that $g_{W}\left(  n\right)  <f_{\omega_{1}}\left(  \left(  n\right)
\right)  $ for almost all $n\in\omega.$ We are in position to prove that
$\mathbb{D}\left(  \mathcal{B}\right)  \models\emptyset\in\overline{A}$ now. It is
enough to prove that $A$ intersects every set of the form $\left(  W\cap
U_{\alpha_{1}}\cap...\cap U_{\alpha_{m}}\right)  _{>k}$ where $W\in
\mathcal{U},$ $\omega_{1}\leq\alpha_{1},...,\alpha_{m}<\omega_{2}$ and
$k\in\omega.$ Find $n\in\omega$ with the following properties:

\begin{enumerate}
\item $g_{W}\left(  n\right)  \neq0.$

\item $g_{W}\left(  n\right)  <f_{\omega_{1}}\left(  \left(  n\right)
\right)  \leq f_{\alpha_{1}}\left(  \left(  n\right)  \right)  ,...,f_{\alpha
_{m}}\left(  \left(  n\right)  \right)  $

\item $n\neq f_{\alpha_{1}}\left(  \emptyset\right)  ,...,f_{\alpha_{m}%
}\left(  \emptyset\right)  .$

\item $k<n.$
\end{enumerate}

Since $g_{W}\left(  n\right)  \neq0,$ there is $s\in\left(  n\right)
^{\frown}\triangle_{g_{W}\left(  n\right)  }\cap W\cap A.$ It follows by Lemma
\ref{Lema prop basicas} that $s\in U_{\alpha_{i}}$ for all $i\leq m.$ Since
$s\left(  0\right)  =n>k,$ we have that $s$ is in $A\cap\left(  W\cap
U_{\alpha_{1}}\cap...\cap U_{\alpha_{m}}\right)  _{>k}.$
\end{proof}

Let us review the proof of the previous theorem. We started with a
$\mathfrak{b}$-scale $\mathcal{D}$ and forced with $\mathbb{L(\mathcal{D})}.$
The proof actually shows that if we complete $\mathcal{D}$ to a scale
$\mathcal{B}$ in any $\omega$-hitting preserving extension such that the first
element of $\mathcal{B}\setminus\mathcal{D}$ is dominating over
$V[A_{\text{\textsf{gen}}}],$ then $\mathcal{B}$ will not be Fr\'{e}chet. Of
course, this does not need to be the case for every $\mathfrak{b}$-scale, so a
more careful approach is required in order to construct a model in which no
$\mathfrak{b}$-scale is Fr\'{e}chet.

We do not know whether there are Fr\'{e}chet scales in the model of Theorem
\ref{Una no Frechet}. In fact, we conjecture that the scale induced by the
Hechler reals is Fr\'{e}chet, but we do not know how to prove it.

\section{There may be no Fr\'{e}chet $\mathfrak{b}$-scale \label{Sec todas no}%
}

In this section, we will prove that it is consistent that no $\mathfrak{b}%
$-scale is Fr\'{e}chet. We follow the approach used in the model
constructed in \cite{MalykhinProblem}. We will perform a finite support
iteration of forcings of the type $\mathbb{L(\mathcal{D})}$ for $\mathcal{D}$
a weak $\mathfrak{b}$-scale. As in \cite{MalykhinProblem}, we will need a
diamond sequence in order to guess an initial segment of a $\mathfrak{b}%
$-scale \textquotedblleft at the right time\textquotedblright. By Lemma
\ref{Lema Laver escala}, its Laver forcing will add $A_{\text{\textsf{gen}}},$
a dense set containing no sequences that converge to $\emptyset.$ The main
difficulty lies in proving that $A_{\text{\textsf{gen}}}$ still accumulate to
$\emptyset,$ even after extending the $\mathfrak{b}$-scale. This is where our
approach diverges from the one taken in \cite{MalykhinProblem}, where the algebraic
structure is used, which is not available in our setting. A completely
different argument is required.

\begin{definition}
Let $\mathcal{F}$ be a filter on the countable set $X$, $\mathbb{\dot{P}}$ an
$\mathbb{L}\left(  \mathcal{F}\right)  $-name for a partial order, and
$\langle\dot{F}_{n}\mid n\in\omega\rangle$ a sequence of $\mathbb{L}\left(
\mathcal{F}\right)  $-names of filters of $\mathbb{P}$ such that
$\mathbb{L}\left(  \mathcal{F}\right)  \Vdash$\textquotedblleft$\mathbb{P=}%
\bigcup\limits_{n\in\omega}\dot{F}_{n}$\textquotedblright. Let $(p,\dot{u}%
)\in\mathbb{L}\left(  \mathcal{F}\right)  \ast\mathbb{\dot{P}}.$

\begin{enumerate}
\item We say that $(p,\dot{u})$ is \emph{suitable }if there is $n\in\omega$
such that $(p,\dot{u})\Vdash$\textquotedblleft$\dot{u}\in\dot{F}_{n}%
$\textquotedblright.

\item A \emph{type }is a pair of the form $\left(  a,m\right)  $ where $a\in
X^{<\omega}$ and $m\in\omega.$

\item We say that $(p,\dot{u})$ is of type $\left(  a,m\right)  $ if
\textsf{st}$\left(  p\right)  =a$ and $p\Vdash$\textquotedblleft$\dot{u}%
\in\dot{F}_{n}$\textquotedblright.

\item Let $\varphi$ be a formula. We say that $\left(  a,m\right)  $
\emph{prefers }\textquotedblleft$\varphi$\textquotedblright\ if there is no
$(q,\dot{v})$ of type $\left(  a,m\right)  $ that forces the negation of
$\varphi.$
\end{enumerate}
\end{definition}

it follows that a condition is suitable if it has a type. Note that the set of
suitable conditions is dense. The following lemma is easy to verify.

\begin{lemma}
Let $\mathcal{F}$ be a filter on a countable set, $\mathbb{\dot{P}}$ an
$\mathbb{L}\left(  \mathcal{F}\right)  $-name for a $\sigma$-filtered
forcing$,\left(  a,m\right)  $ a type and $\varphi$ a formula.
\label{Lema prop preferir}

\begin{enumerate}
\item If $(p_{1},\dot{u}_{1}),...,(p_{n},\dot{u}_{n})$ are of type $\left(
a,m\right)  ,$ then there is $(q,\dot{v})$ of type $\left(  a,m\right)  $ such
that $(q,\dot{v})\leq(p_{1},\dot{u}_{1}),...,(p_{n},\dot{u}_{n}).$

\item If $\left(  a,m\right)  $ prefers\emph{ }\textquotedblleft$\varphi
$\textquotedblright\ and $(p,\dot{u})$ is of type $\left(  a,m\right)  ,$ then
there is $(q,\dot{v})\leq(p,\dot{u})$ such that $(q,\dot{v})\Vdash
$\textquotedblleft$\varphi$\textquotedblright.

\item If $(p,\dot{u})$ is of type $\left(  a,m\right)  $ and $\left(
a,m\right)  $ does not prefer \textquotedblleft$\varphi$\textquotedblright,
then there is $(q,\dot{v})\leq(p,\dot{u})$ of type $\left(  a,m\right)  $ that
forces the negation of $\varphi.$
\end{enumerate}
\end{lemma}

Note that in the second point above, $(q,\dot{v})$ may not be of type $\left(
a,m\right)  $. Another very important property is the following:

\begin{lemma}
[Pure Preference Property]Let $\mathcal{F}$ be a filter on a countable set,
$\mathbb{\dot{P}}$ an $\mathbb{L}\left(  \mathcal{F}\right)  $-name for a
$\sigma$-filtered forcing, $X$ a finite set and $\dot{y}$ an $\mathbb{L}%
\left(  \mathcal{F}\right)  \ast\mathbb{\dot{P}}$-name for an element of $X.$
For every type $\left(  a,m\right)  ,$ there is $z\in X$ such that $\left(
a,m\right)  $ prefers \textquotedblleft$\dot{y}=z$\textquotedblright.
\end{lemma}

\begin{proof}
Assume this is not the case, so for every $z\in X$ there is $(p_{z},\dot
{u}_{z})$ of type $\left(  a,m\right)  $ such that $(p_{z},\dot{u}_{z})\Vdash
$\textquotedblleft$\dot{y}\neq z$\textquotedblright. Find $(q,\dot{v})$ that
extends $(p_{z},\dot{u}_{z})$ for every $z\in X.$ It follows that $(q,\dot
{v})\Vdash$\textquotedblleft$\dot{y}\notin X$\textquotedblright, which is a contradiction.
\end{proof}

We will be working with forcings of the type $\mathbb{L(\mathcal{B})}%
\ast\mathbb{\dot{P}},$ where $\mathcal{B}$ is a weak $\mathfrak{b}$-scale. The
elements of $\mathbb{L(\mathcal{B})}$ are subtrees of $\left(  \omega
^{<\omega}\right)  ^{<\omega}$ that branch into subsets of $\omega^{<\omega}.$
For the convenience of the reader, we adopt the following notational conventions:

\begin{enumerate}
\item Elements of $\omega^{<\omega}$ will be denoted by $s,t$ and $z.$

\item Elements of $\left(  \omega^{<\omega}\right)  ^{<\omega}$ will be
denoted by $a,b$ and $c.$

\item Elements of $\mathbb{L(\mathcal{B})}$ will be denoted by $p,q$ and $r.$

\item Names for elements of $\mathbb{\dot{P}}$ will be denoted by $\dot
{u},\dot{v}$ and $\dot{w}.$
\end{enumerate}

Accordingly, if $p\in\mathbb{L(\mathcal{B})},$ then a typical element of $p$
will be denoted by $a$,$b$ or $c$ and a typical element of \textsf{suc}%
$_{p}\left(  a\right)  $ will be denoted by $s,t$ or $z.$

\begin{proposition}
Let $\mathcal{B}$ a weak $\mathfrak{b}$-scale that is dominating and
Fr\'{e}chet, $\mathbb{\dot{P}}$ an $\mathbb{L}\left(  \mathcal{B}\right)
$-name for a $\sigma$-filtered forcing and $\dot{g}_{1},...,\dot{g}_{m}$ be
$\mathbb{L(\mathcal{B})\ast\dot{P}}$-names for functions from $\omega
^{<\omega}$ to $\omega$ that are dominating. For every $U\in\mathcal{N}%
_{\mathbb{D(\mathcal{B})}}\left(  \emptyset\right)  ,$ we have that:
\label{Prop dificil}

\hfill%
\begin{tabular}
[c]{l}%
$\mathbb{L(\mathcal{B})\ast\dot{P}}\Vdash$\textquotedblleft$U\cap U(\dot
{g}_{1})\cap...\cap U(\dot{g}_{m})\cap\dot{A}_{\text{\textsf{gen}}}%
\setminus\left\{  \emptyset\right\}  \neq\emptyset$\textquotedblright.
\end{tabular}
\qquad\hfill\ \qquad
\end{proposition}

\begin{proof}
We proceed by contradiction. Assume there is a condition $(\overline
{p},\overline{u})$ forcing that the intersection is empty. For convenience, we
may assume that if $a\in\overline{p}\ $extends the stem, then$\ \emptyset
\notin$ \textsf{suc}$_{\overline{p}}\left(  a\right)  $ and the range of $a$
and \textsf{suc}$_{\overline{p}}\left(  a\right)  $ are disjoint. Let $\dot
{g}$ be the name of the function from $\omega^{<\omega}$ to $\omega$ such that
$\dot{g}\left(  s\right)  =$ \textsf{min}$\{\dot{g}_{1}\left(  s\right)
,...,\dot{g}_{m}\left(  s\right)  \}.$ It is easy to see that $\dot{g}$ is
forced to be dominating over $V.$ We may assume there is $\overline{n}%
\in\omega$ such that:

\begin{enumerate}
\item $(\overline{p},\overline{u})\Vdash$\textquotedblleft$\dot{g}_{1}\left(
\emptyset\right)  ,...,\dot{g}_{m}\left(  \emptyset\right)  <\overline{n}%
$\textquotedblright.

\item If $k\geq\overline{n},$ then $\left(  k\right)  \in U.$
\end{enumerate}

Let $a\in\overline{p}$ and $n\in\omega.$ We define the following:

\begin{enumerate}
\item $l\left(  a,n\right)  =$ \textsf{min}$\{k>\overline{n}\mid
a\subseteq\triangle_{k}\}+n.$

\item $M\left(  a,n\right)  $ is the set of all $s\in\omega^{<\omega}$ for
which there is $k\in\omega$ such that $\left(  a,n\right)  $ prefers
\textquotedblleft$\dot{g}\left(  s\right)  <k$\textquotedblright.

\item We say $\left(  a,n\right)  $ \emph{is ugly }if $M\left(  a,n\right)
\setminus\triangle_{l\left(  a,n\right)  }.$

\item Let $(p,\dot{u})\leq(\overline{p},\overline{u}).$ We say that $\left(
a,n\right)  $ \emph{can be realized below} $(p,\dot{u})$ if there is
$(q,\dot{v})\leq(p,\dot{u})$ of type $\left(  a,n\right)  .$
\end{enumerate}

Intuitively, $M\left(  a,n\right)  $ is the collection of all $s\in
\omega^{<\omega}$ such that $\left(  a,n\right)  $ can bound $\dot{g}\left(
s\right)  $ (in term of preference) and a type is ugly if it can bound
something that is \textquotedblleft very far away\textquotedblright. We now
have the following:

\begin{claim}
There is $(p,\dot{u})\leq(\overline{p},\overline{u})$ such that no ugly type
is realized below $(p,\dot{u}).$
\end{claim}

For every $s\in\omega^{<\omega},$ define $X\left(  s\right)  $ as the set of
all types $\left(  a,n\right)  $ such that $s\in M\left(  a,n\right)
\setminus\triangle_{l\left(  a,n\right)  }.$ In other words, $X\left(
s\right)  $ is the set of all types $\left(  a,n\right)  $ such that $s$
testifies that $\left(  a,n\right)  $ is ugly. Note that $X\left(  s\right)  $
is a finite set. This is simply because $s$ is in $\triangle_{l\left(
a,n\right)  }$ for almost all types $\left(  a,n\right)  .$ We can then define
$h:\omega^{<\omega}\longrightarrow\omega$ such that if $s\in\omega^{<\omega}$
and $\left(  a,n\right)  \in X\left(  s\right)  ,$ then $\left(  a,n\right)  $
prefers \textquotedblleft$\dot{g}\left(  s\right)  <h\left(  s\right)
$\textquotedblright\ (in case $X\left(  s\right)  =\emptyset,$ we can simply
take $h\left(  s\right)  =0$).

Since $\dot{g}$ is forced to be a dominating real, there are $k\in\omega$ and
$(p,\dot{u})\leq(\overline{p},\overline{u})$ such that $(p,\dot{u})\Vdash
$\textquotedblleft$h\leq_{k}\dot{g}$\textquotedblright\ and \textsf{st}%
$\left(  p\right)  \nsubseteq\triangle_{k}$. We claim that $(p,\dot{u})$ is as
desired. If this was not true, then there is an ugly type $\left(  a,n\right)
$ that can be realized below $(p,\dot{u}).$ Let $(q,\dot{v})\leq(p,\dot{u})$
that is of type $\left(  a,n\right)  .$ Since \textsf{st}$\left(  q\right)
\nsubseteq\triangle_{k}$ (because \textsf{st}$\left(  p\right)  \subseteq$
\textsf{st}$\left(  q\right)  $), it follows that $l\left(  a,n\right)  >k.$
Since $\left(  a,n\right)  $ is an ugly type, we know that there is $s\in
M\left(  a,n\right)  \setminus\triangle_{l\left(  a,n\right)  }$. Since
$\left(  a,n\right)  \in X\left(  s\right)  ,$ it follows that $\left(
a,n\right)  $ prefers \textquotedblleft$\dot{g}\left(  s\right)  <h\left(
s\right)  $\textquotedblright. In this way, we can find $(r,\dot{w}%
)\leq(q,\dot{v})$ such that $(r,\dot{w})\Vdash$\textquotedblleft$\dot
{g}\left(  s\right)  <h\left(  s\right)  $\textquotedblright. But this is a
contradiction since $s\notin\triangle_{k}$ (recall that $l\left(  a,n\right)
>k$ and $s$ is not even in $\triangle_{l\left(  a,n\right)  }$) and
$(p,\dot{u})\Vdash$\textquotedblleft$h\leq_{k}\dot{g}$\textquotedblright. This
finishes the proof of the claim.

Fix $(p,\dot{u})\leq(\overline{p},\overline{u})$ such that no ugly type is
realized below it. For simplicity, we will assume that \textsf{st}$\left(
p\right)  =\emptyset.$ The argument for the general case is essentially the
same, only requiring much more notation. Let $T$ be the set of all types that
can be realized below $(p,\dot{u}).$ For every type $\left(  a,n\right)  \in
T,$ fix a condition $(p\left(  a,n\right)  ,\dot{u}\left(  a,n\right)
)\leq(p,\dot{u})$ of type $\left(  a,n\right)  $. Note that the stem of
$p\left(  a,n\right)  $ is $a.$

We record some properties of $T.$ For every $\left(  a,n\right)  \in T,$ the
following holds:

\begin{enumerate}
\item $a\in p.$

\item If $b\in p_{\left(  a,n\right)  }$ and $a\subseteq b,$ then $\left(
b,n\right)  \in T.$ In particular, if $z\in$ \textsf{suc}$_{p\left(
a,n\right)  }\left(  a\right)  ,$ then $\left(  a^{\frown}z,n\right)  \in T.$

\item If $q\leq p\left(  a,n\right)  ,$ then $($\textsf{st}$\left(  q\right)
,n)\in T.$
\end{enumerate}

We now have the following:

\begin{claim}
For every $\left(  a,n\right)  \in T$ and $z\in$ \textsf{suc}$_{p\left(
a,n\right)  }\left(  a\right)  \cap U_{>l\left(  a,n\right)  }$ there is
$i\left(  z,a,n\right)  $ such that:

\begin{enumerate}
\item $i\left(  z,a,n\right)  <\left\vert z\right\vert .$

\item There is $j\leq m$ such that $\left(  a^{\frown}z,n\right)  $
prefers \newline\textquotedblleft$\dot{g}_{j}(z\upharpoonright i\left(
z,a,n\right)  )=z\left(  i\left(  z,a,n\right)  \right)  $\textquotedblright.
In particular, $\left(  a^{\frown}z,n\right)  $ prefers \newline%
\textquotedblleft$\dot{g}(z\upharpoonright i\left(  z,a,n\right)  )\leq
z\left(  i\left(  z,a,n\right)  \right)  $\textquotedblright.
\end{enumerate}
\end{claim}

Let $q=p\left(  a,n\right)  _{a^{\frown}\left(  z\right)  }$ (in other words,
$q$ is the tree obtained by adding $z$ to the stem of $p\left(  a,n\right)
$). Note that $q\Vdash$\textquotedblleft$z\in\dot{A}_{\text{\textsf{gen}}}\cap
U$\textquotedblright. Since $U\cap U(\dot{g}_{1})\cap...\cap U(\dot{g}%
_{m})\cap\dot{A}_{\text{\textsf{gen}}}\setminus\left\{  \emptyset\right\}  $
is forced to be empty, Lemma \ref{Lema hoyito} implies that:\medskip

\hfill%
\begin{tabular}
[c]{l}%
$(q,\dot{u}\left(  a,n\right)  )\Vdash$\textquotedblleft$\exists j,k(\dot
{g}_{j}(z\upharpoonright k)=z(k))$\textquotedblright.
\end{tabular}
\hfill\medskip\ 

We can now use the Pure Preference Property to find the exact $j$ and $k$.
This finishes the proof of the claim.

Let $\left(  a,n\right)  \in T$. Define the following items:

\begin{enumerate}
\item $W\left(  a,n\right)  =\{z\upharpoonright i\left(  z,a,n\right)  \mid
z\in$ \textsf{suc}$_{p\left(  a,n\right)  }\left(  a\right)  \cap U_{>l\left(
a,n\right)  }\}.$

\item For every $s\in W\left(  a,n\right)  ,$ denote \newline$Ext_{s}\left(
a,n\right)  =\{z\in$ \textsf{suc}$_{p\left(  a,n\right)  }\left(  a\right)
\cap U_{>l\left(  a,n\right)  }\mid s=z\upharpoonright i\left(  z,a,n\right)
\}.$
\end{enumerate}

We will now prove the following claim:

\begin{claim}
Let $\left(  a,n\right)  \in T$ and $s\in W\left(  a,n\right)  .$ The
following holds:

\begin{enumerate}
\item $s\neq\emptyset$ and $s\left(  0\right)  >l\left(  a,n\right)  .$

\item $W\left(  a,n\right)  $ is infinite.

\item $\mathbb{D}\left(  \mathcal{B}\right)  \models Ext_{s}\left(
a,n\right)  $ is nowhere dense.
\end{enumerate}
\end{claim}

We prove the first point. Pick any $z\in Ext_{s}\left(  a,n\right)  .$ It
follows by definition that $z\left(  0\right)  >l\left(  a,n\right)  .$ In
order to prove that $s\neq\emptyset$ and $s\left(  0\right)  >l\left(
a,n\right)  ,$ it is enough to show that $i\left(  z,a,n\right)  \neq0.$ If it
was the case that $i\left(  z,a,n\right)  =0,$ then we would have that
$\left(  a^{\frown}z,n\right)  $ prefers \textquotedblleft$\dot{g}%
_{j}(\emptyset)=z\left(  0\right)  $\textquotedblright\ for some $j\leq m.$ It
follows that $\left(  a^{\frown}z,n\right)  $ prefers \textquotedblleft%
$\dot{g}_{j}(\emptyset)>\overline{n}$\textquotedblright, which is impossible
(see the properties of $\overline{n}$ at the beginning of the proof).

We now prove the second point. Choose any $k>l\left(  a,n\right)  .$ Since
\textsf{suc}$_{p}\left(  a\right)  $ is dense, we can find $z\in$
\textsf{suc}$_{p}\left(  a\right)  \cap U_{>l\left(  a,n\right)  }$ with
$z\left(  0\right)  =k.$ The conclusion follows by the first point of the claim.

It is time to prove the third point. Let $W\subseteq\omega^{<\omega}$ be a
non-empty open set. We need to find a non-empty open set $U\subseteq W$ that
is disjoint with $Ext_{s}\left(  a,n\right)  .$ If $W\cap\left\langle
s\right\rangle ^{\text{\textsf{c}}}\neq\emptyset,$ then this is a non-empty
set disjoint with $Ext_{s}\left(  a,n\right)  $, since this set is contained
in $\left\langle s\right\rangle .$ Moreover, we may assume there is
$t\in\omega^{<\omega}$ such that $s\subsetneq t$ and $W\subseteq\left\langle
t\right\rangle $ (since $W\subseteq\left\langle s\right\rangle ,$ there is
$t\supsetneq s$ such that $t\in W,$ we then change $W$ for $W\cap\left\langle
t\right\rangle $).

By the first point of the claim, we know that $s\notin\triangle_{l\left(
a,n\right)  }.$ Since $\left(  a,n\right)  $ is not ugly, it follows that
$\left(  a,n\right)  $ does not prefer \textquotedblleft$\dot{g}\left(
s\right)  \leq t\left(  \left\vert s\right\vert \right)  $\textquotedblright.
By Lemma \ref{Lema prop preferir}, we can find $(q,\dot{v})\leq(p\left(
a,n\right)  ,\dot{u}\left(  a,n\right)  )$ of type $\left(  a,n\right)  $ such
that $(q,\dot{v})\Vdash$\textquotedblleft$\dot{g}\left(  s\right)  >t\left(
\left\vert s\right\vert \right)  $\textquotedblright. We may assume that
\textsf{suc}$_{q}\left(  a\right)  $ is an open dense subset of
$\mathbb{D(\mathcal{B})}.$ We claim that $U=W\cap$ \textsf{suc}$_{q}\left(
a\right)  $ (which is non-empty) is disjoint with $Ext_{s}\left(  a,n\right)
.$ Assume there is $z\in U\cap Ext_{s}\left(  a,n\right)  .$ We have the following:

\begin{enumerate}
\item $s\subsetneq t\subseteq z$ (recall that $W\subseteq\left\langle
t\right\rangle $).

\item $z\upharpoonright i\left(  z,a,n\right)  =s$ ($z\in Ext_{s}\left(
a,n\right)  $). In this way, there is $j\leq m$ such that $\left(  a^{\frown
}z,n\right)  $ prefers \textquotedblleft$\dot{g}_{j}\left(  s\right)
=z\left(  \left\vert s\right\vert \right)  =t\left(  \left\vert s\right\vert
\right)  $\textquotedblright. In particular, $\left(  a^{\frown}z,n\right)  $
prefers \textquotedblleft$\dot{g}(s)\leq t\left(  \left\vert s\right\vert
\right)  $\textquotedblright.
\end{enumerate}

Let $(r,\dot{x})\leq(q,\dot{v})$ such that $(r,\dot{x})\Vdash$%
\textquotedblleft$\dot{g}(s)\leq t\left(  \left\vert s\right\vert \right)
$\textquotedblright. But this is impossible since $(q,\dot{v})\Vdash
$\textquotedblleft$\dot{g}\left(  s\right)  >t\left(  \left\vert s\right\vert
\right)  $\textquotedblright. This finishes the proof of the claim.

We need another claim:

\begin{claim}
Let $\left(  a,n\right)  \in T.$ There is $h_{\left(  a,n\right)  }:W\left(
a,n\right)  \longrightarrow\omega^{<\omega}$ with the following properties:

\begin{enumerate}
\item $h_{\left(  a,n\right)  }\left(  s\right)  \in Ext_{s}\left(
a,n\right)  $ (so $s\subsetneq h_{\left(  a,n\right)  }\left(  s\right)  $).

\item \textsf{im}$\left(  h_{\left(  a,n\right)  }\right)  \in$ \textsf{nwd(}%
$\mathbb{D(\mathcal{B})}$\textsf{)}$^{+}.$
\end{enumerate}
\end{claim}

Take an enumeration $W_{\left(  a,n\right)  }=\left\{  s_{i}\mid i\in
\omega\right\}  $ (recall that this set is infinite). Define $N_{0}%
=Ext_{s_{0}}\left(  a,n\right)  $ and $N_{i+1}=Ext_{s_{i}}\left(  a,n\right)
\setminus N_{0}\cup...\cup N_{i}.$ It follows that $\left\{  N_{i}\mid
i\in\omega\right\}  $ is a partition of \textsf{suc}$_{p\left(  a,n\right)
}\left(  a\right)  \cap U_{>l\left(  a,n\right)  }$ into nowhere dense sets.
Since $\mathbb{D(\mathcal{B})}$ is Fr\'{e}chet, Proposition
\ref{nwd no arriba ED} implies that there is $Z\subseteq$ \textsf{suc}%
$_{p\left(  a,n\right)  }\left(  a\right)  \cap U_{>l\left(  a,n\right)  }$
that is not nowhere dense such that $\left\vert Z\cap N_{k}\right\vert \leq1$
for every $k\in\omega$ (we could make sure that $\left\vert Z\cap
N_{k}\right\vert =1$ whenever $N_{k}\neq\emptyset$ if we wanted). We can now
define $h_{\left(  a,n\right)  }:W\left(  a,n\right)  \longrightarrow
\omega^{<\omega}$ as follows: For $s_{i}\in W\left(  a,n\right)  ,$ if $Z\cap
N_{i}\neq\emptyset,$ let $h_{\left(  a,n\right)  }\left(  s_{i}\right)  $ be
the only point in $Z\cap N_{i}.$ In case $Z\cap N_{i}=\emptyset,$ let
$h_{\left(  a,n\right)  }\left(  s_{i}\right)  $ be any element of
$Ext_{s}\left(  a,n\right)  .$ It follows that \textsf{im}$\left(  h_{\left(
a,n\right)  }\right)  $ contains $Z,$ so it is not nowhere dense. This
finishes the proof of the claim.

Let $\left(  a,n\right)  \in T.$ We now define the function $\overline
{h}_{\left(  a,n\right)  }:W\left(  a,n\right)  \longrightarrow\omega
^{<\omega}$ given by $\overline{h}_{\left(  a,n\right)  }\left(  s\right)
=h_{\left(  a,n\right)  }\left(  s\right)  \left(  \left\vert s\right\vert
\right)  $ (this is possible because $s\subsetneq h_{\left(  a,n\right)
}\left(  s\right)  $). Note that if $s\in W\left(  a,n\right)  ,$ then we have
the following:

\begin{enumerate}
\item $h_{\left(  a,n\right)  }\left(  s\right)  \in$ \textsf{suc}$_{p\left(
a,n\right)  }\left(  a\right)  \cap U_{>l\left(  a,n\right)  }$ and
$s\subsetneq h_{\left(  a,n\right)  }\left(  s\right)  .$

\item $\left(  a^{\frown}h_{\left(  a,n\right)  }\left(  s\right)  ,n\right)
$ prefers \textquotedblleft$\dot{g}(s)\leq\overline{h}_{\left(  a,n\right)
}\left(  s\right)  $\textquotedblright\ \newline(this is because $h_{\left(
a,n\right)  }\left(  s\right)  \upharpoonright\left\vert s\right\vert =s$).
\end{enumerate}

Since $\{\overline{h}_{\left(  a,n\right)  }\mid\left(  a,n\right)  \in T\}$
is a countable set and $\mathcal{B}$ is a dominating family, there is
$f\in\mathcal{B}$ such that $\overline{h}_{\left(  a,n\right)  }\leq^{\ast}f$
for every $\left(  a,n\right)  \in T.$ Recall that $\dot{g}$ is forced to be a
dominating real. We can find a suitable $(q,\dot{v})\leq(p,\dot{u})$ and
$k\in\omega$ such that $(q,\dot{v})\Vdash$\textquotedblleft$f<_{k}\dot{g}%
$\textquotedblright. Let $\left(  a,n\right)  $ be the type of $(q,\dot{v}).$
We may assume that $k<l\left(  a,n\right)  .$

\begin{claim}
There is $s\in W\left(  a,n\right)  $ with the following properties:

\begin{enumerate}
\item $h_{\left(  a,n\right)  }\left(  s\right)  \in$ \textsf{suc}$_{q}\left(
a\right)  .$

\item $(q,\dot{v})\Vdash$\textquotedblleft$f\left(  s\right)  <\dot{g}\left(
s\right)  $\textquotedblright.

\item $\overline{h}_{\left(  a,n\right)  }\left(  s\right)  <f\left(
s\right)  .$
\end{enumerate}
\end{claim}

Since \textsf{im}$\left(  h_{\left(  a,n\right)  }\right)  \in$ \textsf{nwd(}%
$\mathbb{D(\mathcal{B})}$\textsf{)}$^{+},$ it follows that \textsf{im}$\left(
h_{\left(  a,n\right)  }\right)  \cap$ \textsf{suc}$_{q}\left(  a\right)  $ is
infinite. Moreover, we know that $\overline{h}_{\left(  a,n\right)  }%
\leq^{\ast}f,$ so we can find $s\in W\left(  a,n\right)  $ such that
$h_{\left(  a,n\right)  }\left(  s\right)  \in$ \textsf{suc}$_{q}\left(
a\right)  $ and $\overline{h}_{\left(  a,n\right)  }\left(  s\right)
<f\left(  s\right)  .$ Since $s\left(  0\right)  >l\left(  a,n\right)  ,$ we
conclude that $s\notin\triangle_{k}$ (recall that $k<l\left(  a,n\right)  ).$
It follows that $(q,\dot{v})\Vdash$\textquotedblleft$f\left(  s\right)
<\dot{g}\left(  s\right)  $\textquotedblright. This finishes the proof of the claim.

We can now finish the proof. Let $s\in W\left(  a,n\right)  $ as above and
$z=h_{\left(  a,n\right)  }.$ Since $\left(  a^{\frown}z,n\right)  $ prefers
\textquotedblleft$\dot{g}(s)\leq\overline{h}_{\left(  a,n\right)  }\left(
s\right)  $\textquotedblright, there is $(r,\dot{w})<(q_{z},\dot{v})$ such
that $(r,\dot{w})\Vdash$\textquotedblleft$\dot{g}(s)\leq\overline{h}_{\left(
a,n\right)  }\left(  s\right)  $\textquotedblright. It follows that
$(r,\dot{w})\Vdash$\textquotedblleft$\dot{g}(s)<f\left(  s\right)
$\textquotedblright. But this is a contradiction because $(q,\dot{v})\Vdash
$\textquotedblleft$f<_{k}\dot{g}$\textquotedblright.
\end{proof}

Let $S_{\omega_{1}}\left(  \omega_{2}\right)  =\{\alpha<\omega_{2}\mid$
\textsf{cof}$\left(  \alpha\right)  =\omega_{1}\}.$ Recall the following principle:

\begin{center}%
\begin{tabular}
[c]{ll}%
$\Diamond(S_{\omega_{1}}\left(  \omega_{2}\right)  )$ & There is $\left\{
D_{\alpha}\mid\alpha\in S_{\omega_{1}}\left(  \omega_{2}\right)  \right\}  $
such that $D_{\alpha}\subseteq\alpha$ for\\
& all $\alpha\in S_{\omega_{1}}\left(  \omega_{2}\right)  $ with the property
that for every\\
& $X\in\omega_{2},$ the set $\left\{  \alpha\mid X\cap\alpha=D_{\alpha
}\right\}  $ is stationary.
\end{tabular}

\end{center}

A sequence as above is called a $\Diamond(S_{\omega_{1}}\left(  \omega
_{2}\right)  )$-\emph{sequence. }It is well known that $\Diamond(S_{\omega
_{1}}\left(  \omega_{2}\right)  )$ holds in the constructible universe of
G\"{o}del (see \cite{Jech} or \cite{Constructibility}). We can now prove the
main theorem of the section, which was inspired in the proof of the main
theorem of \cite{MalykhinProblem}.

\begin{theorem}
It is consistent that no $\mathfrak{b}$-scale is Fr\'{e}chet.
\end{theorem}

\begin{proof}
We start with a model of \textsf{CH + }$\Diamond(S_{\omega_{1}}\left(
\omega_{2}\right)  ).$ Let $\left\{  D_{\alpha}\mid\alpha\in S_{\omega_{1}%
}\left(  \omega_{2}\right)  \right\}  $ be a $\Diamond(S_{\omega_{1}}\left(
\omega_{2}\right)  )$-sequence. Let us construct a finite support iteration
\newline$\langle\mathbb{P}_{\alpha},\mathbb{\dot{Q}}_{\alpha}\mid\alpha
<\omega_{2}\rangle$ such that for every $\alpha<\omega_{2},$ the following holds:

\begin{enumerate}
\item If $\alpha\in$ $S_{\omega_{1}}\left(  \omega_{2}\right)  $ and
$D_{\alpha}$ codes a $\mathbb{P}_{\alpha}$-name for a $\mathfrak{b}$-scale
$\mathcal{B}_{\alpha}$ that is dominating and Fr\'{e}chet, we let
$\mathbb{\dot{Q}}_{\alpha}$ be a $\mathbb{P}_{\alpha}$-name for
$\mathbb{L(\mathcal{B}}_{\alpha}\mathbb{)}.$

\item If $\alpha$ is not as above, let $\mathbb{\dot{Q}}_{\alpha}$ be a
$\mathbb{P}_{\alpha}$-name for Hechler forcing.
\end{enumerate}

It follows that $\mathbb{P}_{\omega_{2}}$ is a \textsf{c.c.c. }forcing that
preserves $\omega$-hitting families. Since at every step we add a dominating
real, it follows that $\mathbb{P}_{\omega_{2}}\Vdash$\textquotedblleft%
$\mathfrak{b=d=c}=\omega_{2}$\textquotedblright\ and by the preservation of
$\omega$-hitting families, $\mathbb{P}_{\omega_{2}}$ forces that there are no
towers of length $\omega_{2}$ (see \cite{BaumgartnerDordalAdjoining} for more
details). Proposition \ref{BaumgartnerDordal} implies that $\mathbb{P}%
_{\omega_{2}}$ forces that every $\mathfrak{b}$-scale is a dominating family.

Let $G\subseteq\mathbb{P}_{\omega_{2}}$ be a generic filter. We shall prove
that there are no Fr\'{e}chet $\mathfrak{b}$-scales in $V\left[  G\right]  .$
Aiming towards a contradiction, assume that there is a Fr\'{e}chet
$\mathfrak{b}$-scale $\mathcal{B}\in V\left[  G\right]  $. By a standard
closing off argument, there is a set $C\subseteq S_{\omega_{1}}\left(
\omega_{2}\right)  $ which is a club relative to $S_{\omega_{1}}\left(
\omega_{2}\right)  $ such that if $\alpha\in C,$ then $V\left[  G_{\alpha
}\right]  \models\mathcal{B}_{\alpha}$ is a dominating Fr\'{e}chet weak
$\mathfrak{b}$-scale, where $G_{\alpha}=G\cap\mathbb{P}_{\alpha}$ and
$\mathcal{B}_{\alpha}=\mathcal{B}\cap V\left[  G_{\alpha}\right]  .$

It follows that there is $\alpha\in C$ such that $D_{\alpha}$ codes
$\mathcal{B}_{\alpha}.$ In this way, we have that $\mathbb{P}_{\alpha
+1}=\mathbb{P}_{\alpha}\ast\mathbb{L(\mathcal{\dot{B}}}_{\alpha}).$ Let
$A_{\text{\textsf{gen}}}\in V\left[  G_{\alpha+1}\right]  $ be the image of
the generic real added by $\mathbb{L(\mathcal{\dot{B}}}_{\alpha}).$

\begin{claim}
$\mathbb{D(\mathcal{B})}\models A_{\text{\textsf{gen}}}$ accumulates to
$\emptyset$ (in $V\left[  G\right]  $).
\end{claim}

It is enough to prove that for every $U\in V\left[  G_{\alpha}\right]  $ basic
open set of $\emptyset$ and $g_{1},...,g_{m}\in\mathcal{B\setminus B}_{\alpha
},$ it is the case that $U\cap U(\dot{g}_{1})\cap...\cap U(\dot{g}_{m}%
)\cap\dot{A}_{\text{\textsf{gen}}}\setminus\left\{  \emptyset\right\}
\neq\emptyset.$ Note that since $V\left[  G_{\alpha}\right]  \models
\mathcal{B}_{\alpha}$ is a dominating family and $g_{1},...,g_{m}%
\in\mathcal{B\setminus B}_{\alpha},$ then they are dominating over $V\left[
G_{\alpha}\right]  .$ Let $\beta>\alpha$ such that $g_{1},...,g_{m}\in
V\left[  G_{\beta}\right]  .$

We go to $V\left[  G_{\alpha}\right]  .$ Let $\mathbb{B}$ be the Boolean
completion of the quotient $\mathbb{P}_{\beta}\diagup G_{\alpha}.$
Propositions \ref{Prop Boolean compl} and \ref{Prop iteracion centered} imply
that $\mathbb{B}$ is $\sigma$-filtered. Proposition \ref{Prop dificil} implies
that $U\cap U(\dot{g}_{1})\cap...\cap U(\dot{g}_{m})\cap\dot{A}%
_{\text{\textsf{gen}}}\setminus\left\{  \emptyset\right\}  \neq\emptyset.$

\begin{claim}
$\mathbb{D(\mathcal{B})}\models A_{\text{\textsf{gen}}}$ does not contain a
sequence converging to $\emptyset$ (in $V\left[  G\right]  $).
\end{claim}

Let $Y$ be an infinite subset of $A_{\text{\textsf{gen}}}.$ Since
$\mathcal{I}_{\mathbb{D(\mathcal{B}}_{\alpha}\mathbb{)}}\left(  \emptyset
\right)  \upharpoonright A_{\text{\textsf{gen}}}$ is an $\omega$-hitting
ideal, there is $B\in\mathcal{I}_{\mathbb{D(\mathcal{B}}_{\alpha}\mathbb{)}%
}\left(  \emptyset\right)  $ such that $Y\cap B$ is infinite. Let $U$ be a
basic neighborhood of $\emptyset$ (in $\mathbb{D(\mathcal{B}}_{\alpha
}\mathbb{)}$) such that $B\cap U=\emptyset.$ It follows that $U\setminus Y$ is
infinite. Since $U$ is still an open set in $\mathbb{D(\mathcal{B})}$, the
result follows.
\end{proof}

\section{The Category Dichotomy for definable ideals\label{CAT}}

In \cite{KatetovOrderonBorelIdeals} the third author proved that if
$\mathcal{I}$ is a Borel ideal, then either $\mathcal{I}\leq_{\text{\textsf{K}%
}}$ \textsf{nwd }or there is $X\in\mathcal{I}^{+}$ such that $\mathcal{ED}%
\leq_{\text{\textsf{K}}}\mathcal{I}\upharpoonright X.$ This statement is known
as the \emph{Category Dichotomy.}\footnote{Note that the two alternatives of
the Category Dichotomy are not mutually exclusive, so it is not a dichotomy in
the traditional sense.} It is natural to ask for which other classes of ideals
the Category Dichotomy can be extended. The article \cite{DicotomiaCat}
contains many results in this direction. Here we are interested in how far the
dichotomy can hold within the projective hierarchy. Using results from the
previous sections, we will provide (consistently) a new example of an ideal
that does not satisfy the Category Dichotomy with the least possible complexity.

In \cite{IdealDichotomiesSolovayModel} the second author and Navarro proved
that the Category Dichotomy holds for analytic ideals, as well as for all
ideals in the Solovay model. We will now prove that the dichotomy also holds
for co-analytic ideals.

Let $\mathcal{I}$ be an ideal on a countable set. The game $\mathcal{G}%
_{\text{\textsf{Cat}}}(\mathcal{I})$ (which was first played by Laflamme, see
\cite{FilterGamesandCombinatorialPropertiesofStrategies} and
\cite{FilterGamesLeary}) is played in the following way:

\begin{center}%
\begin{tabular}
[c]{|l|l|l|l|l|l|l|l|l|}\hline
$\mathsf{I}$ & $A_{0}$ &  & $A_{1}$ &  & $...$ & $A_{0}$ &  & $...$\\\hline
$\mathsf{II}$ &  & $b_{0}$ &  & $b_{1}$ &  &  & $b_{n}$ & \\\hline
\end{tabular}

\end{center}

The game lasts $\omega$ rounds. In round $n,$ Player $\mathsf{I}$ plays
$A_{n}\in\mathcal{I}$ and Player $\mathsf{II}$ responds with $b_{n}\notin
A_{n}.$ \emph{Player} $\mathsf{II}$ \emph{wins the game} if the $\left\{
b_{n}\mid n\in\omega\right\}  \in\mathcal{I}^{+}.$ In this game and the one
below, Player $\mathsf{I}$ will be a man and Player $\mathsf{II}$ a woman. In
the proof of \cite[Theorem 3.1]{KatetovOrderonBorelIdeals}, the third author
obtained the following result.

\begin{proposition}
[H. \cite{KatetovOrderonBorelIdeals}]Let $\mathcal{I}$ be an ideal on
$\omega.$

\begin{enumerate}
\item If for every $X\in\mathcal{I}^{+},$ the Player $\mathsf{II}$ has a
winning strategy in $\mathcal{G}_{\text{\textsf{Cat}}}(\mathcal{I}%
\upharpoonright X),$ then $\mathcal{I}\leq_{\text{\textsf{K}}}$ \textsf{nwd.}

\item If Player $\mathsf{I}$ has a winning strategy in $\mathcal{G}%
_{\text{\textsf{Cat}}}(\mathcal{I}),$ then there is $X\in\mathcal{I}^{+}$ such
that $\mathcal{ED}\leq_{\text{\textsf{K}}}\mathcal{I}\upharpoonright X.$
\end{enumerate}
\end{proposition}

The following is a trivial consequence of this last proposition:

\begin{corollary}
Let $\Gamma$ be a class of ideals such that if $\mathcal{I}\in\Gamma$ and
$X\in\mathcal{I}^{+},$ then $\mathcal{I}\upharpoonright X\in\Gamma.$ If for
every $\mathcal{I}\in\Gamma$ the game $\mathcal{G}_{\text{\textsf{Cat}}%
}(\mathcal{I})$ is determined, then every ideal in $\Gamma$ satisfies the
Category Dichotomy. \label{Determinacion clases}
\end{corollary}

In \cite{TopologicalRepresentations} the fourth author and Sabok defined an
\textquotedblleft unfolded\textquotedblright\ version of the previous game,
which we will now review. We need to introduce some notation:

\begin{enumerate}
\item Let $R$ be the set of all functions $f:\omega\longrightarrow\omega
\cup\left\{  -1\right\}  $ such that $f(n)<n$ for all $n\in\omega$ and the set $\left\{  n\mid f\left(
n\right)  \neq-1\right\}  $ is infinite.

\item Let $f\in R.$ Define $\widetilde{f}$ $:\omega\longrightarrow\omega$ as
the function obtained in the following way: enumerate $\left\{  n\mid f\left(
n\right)  \neq-1\right\}  =\left\{  n_{i}\mid i\in\omega\right\}  $ in an
increasing way and let $\widetilde{f}\left(  i\right)  =f\left(  n_{i}\right)
.$
\end{enumerate}

Intuitively, if we imagine $f\in R$ as an infinite sequence of elements from
$\omega\cup\left\{  -1\right\}  ,$ then $\widetilde{f}$ is the sequence
obtained by deleting all the occurrences of $-1$ and then \textquotedblleft
compressing\textquotedblright\ to get rid of the empty spaces.

Let $\mathcal{I}$ be an ideal on a countable set and $F:\omega^{\omega
}\longrightarrow\mathcal{I}^{+}$ any function. The game
$\mathcal{H}_{\text{\textsf{Cat}}}(\mathcal{I},F)$ is played in the following way:

\begin{center}%
\begin{tabular}
[c]{|l|l|l|l|l|l|l|l|l|}\hline
$\mathsf{I}$ & $A_{0}$ &  & $A_{1}$ &  & $...$ & $A_{n}$ &  & $...$\\\hline
$\mathsf{II}$ &  & $\left(  b_{0},m_{0}\right)  $ &  & $\left(  b_{1}%
,m_{1}\right)  $ &  &  & $\left(  b_{n},m_{n}\right)  $ & \\\hline
\end{tabular}

\end{center}

The game lasts $\omega$ rounds. In round $n,$ Player $\mathsf{I}$ plays
$A_{n}\in\mathcal{I}$ and Player $\mathsf{II}$ responds with $\left(
b_{n},m_{n}\right)  $ such that $b_{n}\notin A_{n}$ and $m_{n}\in
n\cup\left\{  -1\right\}  $ (intuitively, $m_{n}=-1$ can be interpreted as
Player $\mathsf{II}$ refraining from making a move). \emph{Player}
$\mathsf{II}$ \emph{wins the game }if the following conditions are met:

\begin{enumerate}
\item The set $\left\{  n\mid m_{n}\neq-1\right\}  $ is infinite.

\item $F(\widetilde{g})=\left\{  b_{n}\mid n\in\omega\right\}  ,$ where
$g:\omega\longrightarrow\omega\cup\left\{  -1\right\}  $ is the function such
that $g\left(  n\right)  =m_{n}.$
\end{enumerate}

Note that if Player $\mathsf{II}$ was the winner of the game, then $\left\{
b_{n}\mid n\in\omega\right\}  \in\mathcal{I}^{+}.$ The first point of the
following proposition was proved by the fourth author and Sabok as part of the
proof of Theorem 1.6 of \cite{TopologicalRepresentations}. We provide a proof
for completeness, but first we introduce some more notation. Let $n\in\omega.$
Define $T_{n}$ as the set of all $t:n\longrightarrow n\cup\left\{  -1\right\}
$ such that $t\left(  i\right)  <i$ for all $i<n.$ Note that each $T_{n}$ is a finite set. Let $s\in\omega^{n}$ and
$t\in T_{n}.$ Define $s\ast t:n\longrightarrow\omega\times\left(
\omega\cup\left\{  -1\right\}  \right)  $ given by $\left(  s\ast t\right)
\left(  i\right)  =\left(  s\left(  i\right)  ,t\left(  i\right)  \right)  .$

\begin{proposition}
[K., Sabok \cite{TopologicalRepresentations}]Let $\mathcal{I}$ be a co-analytic
ideal on $\omega$ and $F:\omega^{\omega}\longrightarrow\mathcal{I}^{+}$ a
continuous surjection\footnote{Note that the existence of such function is guaranteed by the fact that $\mathcal{I}$ is co-analytic.}. \label{eg de H}

\begin{enumerate}
\item If Player $\mathsf{I}$ has a winning strategy in $\mathcal{H}%
_{\text{\textsf{Cat}}}(\mathcal{I},F),$ then he has a winning strategy in
$\mathcal{G}_{\text{\textsf{Cat}}}(\mathcal{I}).$

\item If Player $\mathsf{II}$ has a winning strategy in $\mathcal{H}%
_{\text{\textsf{Cat}}}(\mathcal{I},F),$ then she has a winning strategy in
$\mathcal{G}_{\text{\textsf{Cat}}}(\mathcal{I}).$
\end{enumerate}
\end{proposition}

\begin{proof}
Let $\sigma:\left(  \omega\times \left(\omega\cup\left\{-1\right\}\right)\right)  ^{<\omega}\longrightarrow\mathcal{I}$
be a a winning strategy for the Player $\mathsf{I}$ in the game $\mathcal{H}%
_{\text{\textsf{Cat}}}(\mathcal{I},F).$ We will use $\sigma$ to define a
strategy $\pi:\omega^{<\omega}\longrightarrow\mathcal{I}$ for him in the game
$\mathcal{G}_{\text{\textsf{Cat}}}(\mathcal{I}).$ Both games start with
$\pi\left(  \emptyset\right)  =\sigma\left(  \emptyset\right)  .$ Assume we
are at round $n+1$ and Player $\mathsf{I}$ has to make a move (in
$\mathcal{G}_{\text{\textsf{Cat}}}(\mathcal{I})$) after Player $\mathsf{II}$
played $s\in\omega^{n+1}.$ Player $\mathsf{I}$ plays (in $\mathcal{G}%
_{\text{\textsf{Cat}}}(\mathcal{I})$) the set $\pi\left(  s\right)
=\bigcup\limits_{t\in T_{n+1}}\sigma\left(  s\ast t\right)  \in\mathcal{I}.$
We claim that this is a winning strategy for Player $\mathsf{I}$ in
$\mathcal{G}_{\text{\textsf{Cat}}}(\mathcal{I}).$

Assume this is not the case, so there is a run of the game $\mathcal{G}%
_{\text{\textsf{Cat}}}(\mathcal{I})$ where Player $\mathsf{I}$ followed $\pi,$
but Player $\mathsf{II}$ was declared the winner. Let $\left(  b_{i}\right)
_{i\in\omega}$ be the sequence played by Player $\mathsf{II.}$ Since she was
the winner, we know that $Y=\left\{  b_{i}\mid i\in\omega\right\}
\in\mathcal{I}^{+}.$ Recall $F$ is onto, so $Y$ is in the range of $F.$
Moreover, we can find $g\in R$ such that $F(\widetilde{g})=Y.$ Consider the
run of the game $\mathcal{H}_{\text{\textsf{Cat}}}(\mathcal{I},F)$ where
Player $\mathsf{I}$ followed $\sigma$ and Player $\mathsf{II}$ played $\left(
b_{n},g\left(  n\right)  \right)  $ at round $n.$ This is a valid run of the
game since for every $n\in\omega,$ it is the case that $g(n)\in n\cup\left\{-1\right\}$ and $\sigma\left(
(b_{0},g\left(  0\right)  ),...,\left(  b_{n,}g\left(  n\right)  \right)
\right)  $ is contained in $\pi\left(  b_{0},...,b_{n}\right)  ,$ so $b_{n+1}$
is not in $\sigma\left(  (b_{0},g\left(  0\right)  ),...,\left(
b_{n,}g\left(  n\right)  \right)  \right)  $ (it is not even in $\pi\left(
b_{0},...,b_{n}\right)  ,$ which is bigger). It follows that Player
$\mathsf{II}$ won the game, but this was impossible since $\sigma$ was a
winning strategy.

We now prove the second part of the proposition. Player $\mathsf{II}$ can
easily win $\mathcal{G}_{\text{\textsf{Cat}}}(\mathcal{I})$ by playing the
first coordinates of her winning strategy of $\mathcal{H}_{\text{\textsf{Cat}%
}}(\mathcal{I},F).$
\end{proof}

We now prove the determinacy of these games for co-analytic ideals.

\begin{proposition}
Let $\mathcal{I}$ be a co-analytic ideal on $\omega$ and $F:\omega^{\omega
}\longrightarrow\mathcal{I}^{+}$ a continuous surjection. The games
$\mathcal{H}_{\text{\textsf{Cat}}}(\mathcal{I},F)$ and $\mathcal{G}%
_{\text{\textsf{Cat}}}(\mathcal{I})$ are determined.
\label{Juegos determinados}
\end{proposition}

\begin{proof}
By Proposition \ref{eg de H}, the determinacy of $\mathcal{H}%
_{\text{\textsf{Cat}}}(\mathcal{I},F)$ implies the determinacy of
$\mathcal{G}_{\text{\textsf{Cat}}}(\mathcal{I}).$ We now show that
$\mathcal{H}_{\text{\textsf{Cat}}}(\mathcal{I},F)$ is determined. By Martin's
Determinacy Theorem for Borel games (see \cite{Kechris}), it is enough to
prove that $\left\{  (B,g)\mid g\in R\wedge F\left(  \widetilde{g}\right)
=B\right\}  \subseteq\mathcal{P}\left(  \omega\right)  \times\left(
\omega\cup\left\{  -1\right\}  \right)  ^{\omega}$ (the winning set for Player
$\mathsf{II}$) is Borel. This is easy: membership in $R$ is $G_{\delta}$ while
the condition $F\left(  \widetilde{g}\right)  =B$ is closed by the continuity
of $F.$
\end{proof}

Since the restriction of a co-analytic ideal is also co-analytic, by Proposition
\ref{Juegos determinados} and \ref{Determinacion clases}, we obtain the following:

\begin{theorem}
Every co-analytic ideal satisfies the Category Dichotomy.
\end{theorem}

How far can the Category Dichotomy be extended? If we assume suitable
determinacy axioms or instances of the \emph{Closed-sets Covering Property}
(see \cite{CoveringSolecki}, \cite{PerfectSetPropertiesinL(R)} and
\cite{IdealDichotomiesSolovayModel}), we can extend the Category Dichotomy
throughout the entire Projective Hierarchy. However, just in \textsf{ZFC, }the
class of analytic and co-analytic ideals is the highest we can go, since it is
consistent that the Category Dichotomy fails for $\triangle_{2}^{1}$
ideals.\textsf{ }

We were able to find two consistent examples of $\triangle_{2}^{1}$ ideals in the
literature for which the Category Dichotomy fails:

\begin{enumerate}
\item \emph{The ideal generated by a co-analytic tight MAD family. }In
\cite{coanalyticWitnesses}, Bergfalk, Fischer and Switzer proved that $V=L$
implies that there is a co-analytic tight \textsf{MAD }family. The ideal
generated by such family is a $\triangle_{2}^{1}$ ideal and can not satisfy
the Category Dichotomy (see Proposition 24 of \cite{DicotomiaCat}).

\item \emph{The dual of a }$\triangle_{2}^{1}$\emph{ Ramsey ultrafilter. }In \cite[Theorem 5.1]{coanalyticUltrafiltersBases}, Schilhan proved that it is
consistent that there is a $\triangle_{2}^{1}$\emph{ }Ramsey ultrafilter. The
dual of such ultrafilter cannot satisfy the Category Dichotomy (see \cite[Proposition 23]{DicotomiaCat}).
\end{enumerate}

We will now provide a new example of a $\triangle_{2}^{1}$ ideal for which the
Category Dichotomy fails using Dow spaces. The following is 
\cite[Theorem 45]{DicotomiaCat}:

\begin{theorem}
[Dow, F., G., H. \cite{DicotomiaCat}] If $X$ is a topological space with the
following properties:

\begin{enumerate}
\item $X$ is countable.

\item $X$ is zero dimensional.

\item $X$ is Fr\'{e}chet.

\item $X$ has uncountable $\pi$-character everywhere.
\end{enumerate}

Then \textsf{nwd}$\left(  X\right)  $ does not satisfy the Category Dichotomy.
\end{theorem}

It follows that if $\mathcal{B}$ is a Fr\'{e}chet $\mathfrak{b}$-scale, then
the ideal \textsf{nwd($\mathbb{D(\mathcal{B})}$) }does not satisfy the
Category Dichotomy. We now prove the following:

\begin{theorem} \label{CatUnderV=L}
If $V=L$, then there exists a Fr\'echet $\mathfrak{b}$-scale $\mathcal{B}$ such that
$\mathsf{nwd}\bigl(\mathbb{D}(\mathcal{B})\bigr)$ is a $\Delta^1_2$ ideal.
\end{theorem}

The following lemma is used in the proof of the preceding theorem.
  
\begin{lemma} \label{Lemma_V=L}
Let $(X,\tau)$ be a countable topological space, and suppose that $\mathscr{B}$ is a co-analytic basis for $\tau$. Then the topology $\tau$ is $\Sigma^1_2$, and the ideal $\mathsf{nwd}(X,\tau)$ is $\Delta^1_2$.
\end{lemma}
\begin{proof} \,
As $\mathscr{B}$ is co-analytic, for each $s \in X$ the set
\[
X_s=\bigl\{ A \subseteq X : s \in A \;\rightarrow\; \exists U \in \mathscr{B}
\bigl( s \in U \wedge U \subseteq A \bigr) \bigr\}
\]
is $\Sigma^1_2$.
Moreover, for any $A \subseteq X$
\[
A \in \tau \;\Longleftrightarrow\;
\forall s \in X \left(
A \in X_s
\right).
\]
Therefore $\tau$ belongs to $\Sigma^1_2$, as it is a countable intersection of
$\Sigma^1_2$ sets.
    
Recall that $\triangle_{2}^{1}$ denotes the class $\Sigma^1_2 \cap \Pi^1_2$.
We first show that $\mathsf{nwd}(X, \tau)$ is a $\Pi^1_2$ subset of $\mathscr{P}(X)$. To that end, it suffices to verify that its complement
\[
\mathsf{nwd}(X, \tau)^+ = \mathscr{P}(X) \setminus \mathsf{nwd}(X, \tau)
\]
is $\Sigma^1_2$. This follows from the following equivalence:
\[
A \in \mathsf{nwd}(X, \tau)^+  \Longleftrightarrow
\exists U \in \mathscr{B}\setminus\{\emptyset\}
\;\forall V \in \mathscr{B}\setminus\{\emptyset\}\;
(V \subseteq U \rightarrow V \cap A \neq \emptyset)
\]
which defines a $\Sigma^1_2$ condition under the assumption that $\mathscr{B}$ is co-analytic. Since the complement of a $\Sigma^1_2$ set is a $\Pi^1_2$ set, the result follows.

On the other hand, we show that $\mathsf{nwd}(X,\tau)$ is also a $\Sigma^1_2$ subset of $\mathscr{P}(X)$.
We begin with the case of closed nowhere dense sets. Let $\overline{\mathsf{nwd}}(X,\tau)$ denote the family of closed nowhere dense subsets of $X$. 
We claim that $\overline{\mathsf{nwd}}(X,\tau)$ is $\Sigma^1_2$.
Indeed, a set $F \subseteq X$ is closed and nowhere dense if and only if
its complement $X\setminus F$ is a dense open set. Equivalently,
\[
F \in \overline{\mathsf{nwd}}(X,\tau)
\;\Longleftrightarrow\;
X\setminus F \in \tau
\;\wedge\;
\forall U \in \mathscr{B}\setminus\{\emptyset\}\,
\bigl(U \cap (X\setminus F) \neq \emptyset\bigr).
\]
Since the topology $\tau$ is $\Sigma^1_2$ and $\mathscr{B}$ is a co-analytic, the universal quantification over
$\mathscr{B}\setminus\{\emptyset\}$ defines a co-analytic set.
Therefore, the conjunction above is $\Sigma^1_2$, and it follows that
$\overline{\mathsf{nwd}}(X,\tau)$ is $\Sigma^1_2$.

Now, we consider the general case. A subset
$A \subseteq X$ is nowhere dense if and only if it is contained in some
closed nowhere dense set. That is,
\[
A \in \mathsf{nwd}(X,\tau)
\;\Longleftrightarrow\;
\exists F \in \overline{\mathsf{nwd}}(X,\tau)\,
\bigl(A \subseteq F\bigr)
\]
from which follows that
$\mathsf{nwd}(X,\tau)$ is also a $\Sigma^1_2$ subset of
$\mathscr{P}(X)$.   
\end{proof}

\begin{proof}[Proof of Theorem 52] 
By Lemma~\ref{Lemma_V=L}, it suffices to construct a Fr\'{e}chet \(\mathfrak{b}\)-scale such that the topology \(\tau_{\mathcal{B}}\) admits a co-analytic basis. Note that by Theorem \ref{AllAreFrechet}, $V=L$ implies that all $\mathfrak{b}$-scales are Fr\'{e}chet.

\begin{claim}
    $V=L$ implies that there is a co-analytic $\mathfrak{b}$-scale.
\end{claim}
\begin{proof}
Consider
\[
F \subseteq 
\left(\omega^{\omega^{<\omega}}\right)^{\le \omega}
\times 
\omega^{\omega^{<\omega}}
\times 
\omega^{\omega^{<\omega}}
\]
defined by declaring that $(A,g,f)\in F$ if and only if the following
conditions hold:

\begin{enumerate}
\item $f$ is increasing from ${\omega^{<\omega}}$ to $\omega$, that is, for $s, t \in \omega^{<\omega}$ and $i, j \in \omega$:
\begin{itemize}
        \item If $s \subset t$, then $f(s) < f(t)$, and
        \item If $i < j$, then $f(s^\smallfrown i) < f(s^\smallfrown j)$.
    \end{itemize}

\item For every $h\in\operatorname{ran}(A)$ we have $h <^{*} f$.

\item $g <^{*} f$.
\item $f(\emptyset) > g(\emptyset)$.
\end{enumerate}

Note that $F$ is Borel, since it is defined as a finite conjunction of Borel conditions. Moreover, for every $(A, g) \in \left(\omega^{\omega^{<\omega}}\right)^{\le \omega}
\times 
\omega^{\omega^{<\omega}}
$, the section 
\[
F_{(A,g)}=\left\{f\in\omega^{\omega^{<\omega}} : (A,g,f)\in F\right\}
\]
is cofinal in the Turing degrees\footnote{We refer to elements of $2^\omega$, $\omega^\omega$, or $\omega^{\omega^{<\omega}}$ collectively as \textit{reals}. For $x$ and $y$ reals, we say $x$ is \textit{Turing reducible} to $y$, denoted by $x \le_T y$, if there exists an oracle Turing machine that computes $x$ given $y$ as an oracle. A set of reals $A$ is \textit{cofinal} in the Turing degrees if for every $x \in 2^\omega$ there exists some $y \in A$ such that $x \le_T y$.}.

We will prove this for $A$ countably infinite. The argument for $A$ finite is essentially the same. Fix $\{h_i : i \in \omega\}$ an enumeration of $A$, and an arbitrary real $y\in 2^\omega$. We construct a function $f \in F_{(A,g)}$ such that $y \le_T f$.  

First, define a function $f_0 \in F_{(A, g)}$ by recursion on $\omega^{<\omega}$.
For each $s \in \omega^{<\omega}$, choose $f_0(s)$ so large that:

\begin{enumerate}
\item If $t \subset s$, then $f_0(t) < f_0(s)$, 
\item If $i < j$, then $f_0(s^\smallfrown i) < f_0(s^\smallfrown j),$ and
\item $f_0(s) > \max\left\{h_i(s): i \leq |s| \right\} + g(s)$,
\end{enumerate}
Since only finitely many inequalities must be satisfied at each stage, this construction is possible.

Now, we code $y$ along the branch $\langle 0\rangle, \langle 0,0\rangle, \dots$. Define $f$ by
\[
f(s) = 
\begin{cases}
2f_0(s) + 1 & \text{if } s \text{ is not of the form } \langle 0^n \rangle \text{ for any } n,\\[4pt]
2f_0(\langle 0^n \rangle) + y(n) & \text{if } s = \langle 0^n \rangle \text{ for some } n,
\end{cases}
\]  
where $\langle 0^n \rangle$ denotes the sequence of $n$ many zeros. Note that $f \in F_{(A,g)}$ and $y$ is computable from $f$ since for each $n$,  
\[
y(n) = f(\langle 0^n \rangle) \bmod 2,
\]  
and the parity of $f$ at these nodes can be obtained by an oracle Turing machine with input $f$. Thus $y \le_T f$, establishing that $F_{(A,g)}$ is cofinal in the Turing degrees.

Therefore, it follows from \cite[Theorem 1.3.]{ZoltanVidnyanszky2014} that there exists a
co-analytic set that is compatible with $F$, that is, there is a co-analytic $\mathfrak{b}$-scale 
\end{proof}

Now, let $\mathcal{B} = \{f_\alpha : \alpha < \omega_1\}$ be a co-analytic $\mathfrak{b}$-scale. Note that by construction, $\mathcal{B}$ is well-ordered by $<^*$, that is, for any distinct $f, g \in \mathcal{B}$, either $f <^* g$ or $g <^* f$.
We shall prove that the topology $\tau_{\mathcal{B}}$ admits a co-analytic basis $\mathscr{B}$.
We verify the following statements.
\begin{enumerate}
\item The family 
$\mathcal{U}=\{U(f) : f \in \mathcal{B}\}$
is co-analytic.

Let $\Phi: \omega^{\omega^{<\omega}} \to 2^{\omega^{<\omega}}$ be the mapping defined by $f \mapsto U(f)$. 
We claim that $\Phi$ is continuous. We need to show that the preimage of any subbasic open set is open.

For $s \in \omega^{<\omega}$, consider 
$V_s = \{ A \subseteq \omega^{<\omega} : s \in A \}$ a subbasic open set. Based on the definition of $U(f)$, we have:
\[
f \in \Phi^{-1}(V_s) \iff s \in U(f) \iff \forall t \subsetneq s \left( f(t) \neq s(|t|) \right).
\]
This preimage is a finite intersection of open sets:
\[
\Phi^{-1}(V_s) = \bigcap_{t \subsetneq s} \{ f \in \omega^{\omega^{<\omega}} : f(t) \neq s(|t|) \}.
\]
Therefore, $\Phi^{-1}(V_s)$ is open.
Similarly, for the complement of the subbasic open set:
\[
f \in \Phi^{-1}(2^{\omega^{<\omega}} \setminus V_s) \iff s \notin U(f) \iff \exists t \subsetneq s \left( f(t) = s(|t|) \right).
\]
This preimage is a finite union of open sets:
\[
\Phi^{-1}(2^{\omega^{<\omega}} \setminus V_s) = \bigcup_{t \subsetneq s} \{ f \in \omega^{\omega^{<\omega}} : f(t) = s(|t|) \},
\]
which is also open. Since the preimages of subbasic sets are open, we conclude that $\Phi$ is continuous. Moreover, $\Phi$ is injective. Indeed, let $f, g \in \omega^{\omega^{<\omega}}$ with $f \neq g$. For the minimal $t$ such that $f(t) \neq g(t)$, define $s = t^\frown f(t)$. Then $s \in U(g) \setminus U(f)$, so $\Phi(f) \neq \Phi(g)$.
By Theorem 2.6 in \cite{ivorra2021teoria}, the injective continuous image of a co-analytic set is co-analytic. Therefore, we conclude that $\mathcal{U} = \{U(f) : f \in \mathcal{B}\}$ is co-analytic.


\item 
For each $n \in \omega$, 
the following family is co-analytic
\[ \mathcal{U}^n = \left\{ \bigcap_{f \in C} U(f) : C \in \mathcal{B}^n \right\}. \]

Let $X = \omega^{\omega^{<\omega}}$ and consider the mapping $\Phi_n: X^n \to 2^{\omega^{<\omega}}$ defined by $(f_0, \dots, f_{n-1}) \mapsto \bigcap_{i < n} U(f_i)$. 

First, we show that $\Phi_n$ is continuous. For each $j < n$, let $\pi_j: X^n \to X$ be the projection onto the $j$-th coordinate, and let $\Phi: X \to 2^{\omega^{<\omega}}$ be the continuous map $f \mapsto U(f)$. Then the map $\phi_j = \Phi \circ \pi_j$ is continuous as a composition of continuous functions. 

Note that $\Phi_n$ is given by the finite intersection:
\[
\Phi_n(f_0, \dots, f_{n-1}) = \bigcap_{j < n} \phi_j(f_0, \dots, f_{n-1}).
\]
Since the intersection of $n$ elements in $2^{\omega^{<\omega}}$ is a continuous operation, it follows that $\Phi_n$ is continuous.
Next, we establish that $\Phi_n$ is injective. Let $A, B \in X^n$ be distinct as sets. Let $t \in \omega^{<\omega}$ be a minimal node such that the sets of values $V_A = \{f(t) : f \in A\}$ and $V_B = \{g(t) : g \in B\}$ are distinct, say $V_A \setminus V_B\neq\emptyset$. Suppose $k \in V_A \setminus V_B$ and let $s = t^\frown k$. Thus $s \in \Phi_n(B) \setminus \Phi_n(A)$, so it follows that $\Phi_n(A) \neq \Phi_n(B)$.
Finally, note that the product $\mathcal{B}^n$ is co-analytic in $X^n$. By Theorem 2.6 in \cite{ivorra2021teoria}, we conclude that $\Phi_n(\mathcal{B}^n)= \left\{ \bigcap_{f \in C} U(f) : C \in \mathcal{B}^n \right\}$ is co-analytic.

\item The family $\mathcal{U}^{<\omega}=\left\{ \bigcap_{f \in C} U(f) : C \in \mathcal{B}^{<\omega} \right\}$ is co-analytic.

Since the family $\mathcal{U}^{<\omega}$ can be expressed as the countable union of the families $\mathcal{U}^n$ previously defined:
\[
\mathcal{U}^{<\omega} = \bigcup_{n \in \omega} \mathcal{U}^n = \bigcup_{n \in \omega} \left\{ \bigcap_{f \in C} U(f) : C \in \mathcal{B}^n \right\},
\]
and each $\mathcal{U}^n$ is co-analytic, it follows that $\mathcal{U}^{<\omega}$ is a co-analytic family.

\item For any finite sets $A, B \subset \omega^{<\omega}$, the family 
\[
\mathcal{X}_{A, B} = \left\{ \bigcap_{f \in C} U(f)  \cap  \bigcap_{s \in A} \langle s \rangle  \cap  \bigcap_{s \in B} \langle s \rangle^c  : C \in [\mathcal{B}]^{<\omega} \right\}
\]
is co-analytic.

Let $K_{A, B} = \left( \bigcap_{s \in A} \langle s \rangle \right) \cap \left( \bigcap_{s \in B} \langle s \rangle^c \right)$.
Note that if $K_{A, B}=\emptyset$, then $\mathcal{X}_{A, B} = \{ \emptyset \}$, which is a closed set and thus co-analytic. We assume therefore that $K_{A, B} \neq \emptyset$. In particular, we assume that $A$ and $B$ are disjoint. 
Let $\Psi_{A, B}: 2^{\omega^{<\omega}} \to 2^{\omega^{<\omega}}$ be the mapping defined by $\Psi_{A, B}(U) = U \cap K_{A, B}$. The continuity of $\Psi_{A, B}$ follows from the fact that the pre-image of any sub-basic open set, that is, for each $s \in \omega^{<\omega}$ the set $V_s = \{ A \subseteq \omega^{<\omega} : s \in A \}$ and its complement $V_s^c$, is either $V_s, \emptyset, V_s^c$, or $2^{\omega^{<\omega}}$, all of which are open.

Let $\mathcal{U}^{<\omega}$ be the co-analytic family of (3). We partition this family into
$\mathcal{U}_{\emptyset} = \{ U \in \mathcal{U}^{<\omega} :  U \cap K_{A, B} = \emptyset \}$ and $\mathcal{U}_{+} = \mathcal{U}^{<\omega} \setminus \mathcal{U}_{\emptyset}$. Note that $\mathcal{U}_{+} = \mathcal{U}^{<\omega} \cap \{ U : U \cap K_{A, B} \neq \emptyset \}$ remains co-analytic, since the condition $U \cap K_{A, B} \neq \emptyset$ defines an open set in $2^{\omega^{<\omega}}$, and $\mathcal{U}^{<\omega}$ is co-analytic.

On the domain $\mathcal{U}_{+}$, the mapping $\Psi_{A, B}$ is injective. 
To see that, consider $U_1, U_2 \in \mathcal{U}_{+}$ such that  $U_1 \neq U_2$. Let $C_1, C_2 \in \mathcal{B}^{<\omega}$ be such that $U_1=\bigcap \{U(f) : f \in C_1\}$ and $U_2=\bigcap \{U(f) : f \in C_2\}$. Therefore, there is $g \in C_1 \Delta C_2$ and, without loss of generality, we assume $g \in C_1 \setminus C_2$. Since $\mathcal{B}$ is a $\mathfrak{b}$-scale, there exists $N \in \omega$ such that for all $s \in \omega^{<\omega}$ with $|s| > N$, the value $g(s)$ is distinct from $f(s)$ for every $f \in C_2$. This implies that for any such $s$, if $s\in U_2$ then the successor node $s^* = s^\smallfrown g(s)$ belongs to $U_2 \setminus U_1$. Given that $U_2 \cap K_{A, B}$ is non-empty and the trees in $\mathcal{U}_{+}$ possess an infinite branching structure, there are infinitely many nodes $s \in U_2 \cap K_{A, B}$ with $|s| > N$. For each such $s$, the successor $s^*$ remains within the cones generated by $A$. Furthermore, the finiteness of $B$ ensures that we can choose at least one $s$ such that $s^* = s^\smallfrown g(s) \in K_{A,B}$. Thus, $\Psi_{A, B}(U_1) \neq \Psi_{A, B}(U_2)$.

By Theorem 2.6 in \cite{ivorra2021teoria}, we conclude that $\Psi_{A, B}(\mathcal{U}_{+})$ is co-analytic. Since $\Psi_{A, B}(\mathcal{U}_{\emptyset}) = \{\emptyset\}$ is a closed set, it follows that
\[
\mathcal{X}_{A, B} = \Psi_{A, B}(\mathcal{U}_{\emptyset}) \cup \Psi_{A, B}(\mathcal{U}_{+})
\]
is co-analytic.

\item The family $\mathscr{B}$ generated by the cones $\langle s \rangle$ and co-cones $\langle s \rangle^c$ for $s \in \omega^{<\omega}$, together with the trees $U(f)$ for $f \in \mathcal{B}$, is co-analytic. 

Note that for any $U \in 2^{\omega^{<\omega}}$
$$U \in \mathscr{B} \iff \exists A, B \in [\omega^{<\omega}]^{<\omega} \, (U \in \mathcal{X}_{A, B}).$$
As each $\mathcal{X}_{A, B}$ is co-analytic and the class of co-analytic sets is closed under countable unions, it follows that $\mathscr{B}$ is co-analytic.

\end{enumerate}

This completes the proof that $\tau_{\mathcal{B}}$ admits a co-analytic basis.

\end{proof}

\section{Open Questions \label{Sec pregunts}}

In this final section, we present several open questions that we have been
unable to solve. We first restate Moore's problem:

\begin{problem}
[Moore]Is there a countable, Fr\'{e}chet, zero dimensional space with $\pi
$-weight exactly $\mathfrak{b}$?
\end{problem}

In the paper, we proved that it is consistent that every $\mathfrak{b}$-scale
is Fr\'{e}chet and also that none is. However, the following remains unsolved:

\begin{problem}
Is it consistent that there are Fr\'{e}chet $\mathfrak{b}$-scales and a
non-Fr\'{e}chet $\mathfrak{b}$-scales at the same time?
\end{problem}

We conjecture that this happens in the model of Theorem \ref{Una no Frechet}.

\begin{problem}
Is there a Fr\'{e}chet $\mathfrak{b}$-scale in the model of Theorem
\ref{Una no Frechet}?
\end{problem}

We do not know much about the ideal \textsf{nwd(}$\mathbb{D}\left(  \mathcal{D}%
\right)  $) in case $\mathcal{D}$ is not a Fr\'{e}chet $\mathfrak{b}$-scale.

\begin{problem}
\qquad\ \ \qquad\qquad\qquad\ \ \ \ \qquad\ \ 

\begin{enumerate}
\item Is it consistent that there is a $\mathfrak{b}$-scale $\mathcal{D}$ such
that \textsf{nwd(}$\mathbb{D}\left(  \mathcal{D}\right)  $) satisfies the
Category Dichotomy?

\item Is it consistent that for every $\mathfrak{b}$-scale $\mathcal{D},$ the
ideal \textsf{nwd(}$\mathbb{D}\left(  \mathcal{D}\right)  $) satisfies the
Category Dichotomy?
\end{enumerate}
\end{problem}

Finally, the following problem may also be of interest.

\begin{problem}
How do the combinatorial properties of a scale reflect on the topological
properties of its Dow space and vice versa?
\end{problem}

{\normalsize
\bibliographystyle{plain}
\bibliography{Bibliografia}
}\ 

\bigskip

Ra\'{u}l Figueroa-Sierra

Departamento de Matem\'{a}ticas, Universidad de Los Andes (Bogotá)

r.figueroa@uniandes.edu.co

\qquad\ \qquad\ \ \qquad\qquad\qquad\qquad\ \ \ \ \ \ 

Osvaldo Guzm\'{a}n

Centro de Ciencias Matem\'{a}ticas, UNAM.

oguzman@matmor.unam.mx

\qquad\qquad\qquad\ \ \ \ \ \ \ \qquad\qquad\qquad\ \ \ \ \ \ 

Michael Hru\v{s}\'{a}k

Centro de Ciencias Matem\'{a}ticas, UNAM.

michael@matmor.unam.mx

\qquad\qquad

Adam Kwela

Faculty of Mathematics, Physics and Informatics, University of Gda\'{n}sk

Adam.Kwela@ug.edu.pl

\end{document}